\documentclass[11pt]{article}

\usepackage{amsfonts}
\usepackage{amssymb,amsmath,amsthm}
\usepackage{latexsym}
\usepackage{fullpage}
\usepackage{hyperref}
\usepackage{graphicx}
\usepackage{tkz-euclide,subfigure}
\usepackage{tikz}
\usepackage[utf8]{inputenc}
\usepackage{pgfplots}
\usetikzlibrary{shapes.misc}
\usepackage{amssymb}

\newtheorem{theorem}{Theorem}[section]

\newtheorem{prop}[theorem]{Proposition}

\newtheorem{lemma}[theorem]{Lemma}
\newtheorem{cor}[theorem]{Corollary}
\newtheorem{definition}[theorem]{Definition}

\newtheorem{example}[theorem]{Example}

\theoremstyle{definition}
\newtheorem{remark}[theorem]{Remark}

\newcounter{tenumerate}

\renewcommand{\epsilon}{\varepsilon}

\newcommand{\remove}[1]{}

\renewcommand{\leq}{\leqslant}
\renewcommand{\geq}{\geqslant}

\def\XXint#1#2#3{{\setbox0=\hbox{$#1{#2#3}{\int}$}
\vcenter{\hbox{$#2#3$}}\kern-.5\wd0}}

\title{\textbf{Representation theorems for dynamic convex risk measures}}
\begin{document}
\author{{Shiqiu Zheng\thanks{E-mail: shiqiu@tust.edu.cn}}\\
  \\
\small College of Sciences, Tianjin University of Science and Technology, Tianjin 300457, China}
\date{}
\maketitle
\begin{abstract}
In this paper, we prove that under the domination condition:
\begin{equation*}
{\cal{E}}^{-\mu,-\nu}[-\xi|{\cal{F}}_t]\leq\rho_t(\xi)\leq{\cal{E}}^{\mu,\nu}[-\xi|{\cal{F}}_t],\quad \forall\xi\in \mathcal{L}^{\exp}_T\ (\text{resp.}\ L^2(\mathcal{F}_T)),\ \forall t\in[0,T],
\end{equation*}
where ${\cal{E}}^{\mu,\nu}$ is the $g$-expectation with generator $\mu|z|+\nu|z|^2, \mu\geq0, \nu\geq0$, the dynamic convex (resp. coherent) risk measure $\rho$ admits a representation as a $g$-expectation, whose generator $g$ is convex (resp. sublinear) in the variable $z$ and has a quadratic (resp. linear) growth. As an application, we show that such dynamic convex (resp. coherent) risk measure $\rho$ admits a dual representation, where the penalty term (resp. the set of probability measures) is characterized by the corresponding generator $g$.
\end{abstract}
\section{Introduction}
An important task in risk management is to quantify the risks of financial positions. Artzner et al. \cite{AD} introduced the notion of coherent risk measures, which are defined by a set of four desirable properties for measures of risk. This seminal framework was later extended to convex risk measures by F\"{o}llmer and Schied \cite{FS} and Frittelli and Rosazza Gianin \cite{FG}. However, coherent and convex risk measures are static, meaning that they are not updated as new information becomes available. This motivates the development of dynamic risk measures (DRMs), including dynamic coherent risk measures (coherent DRMs) (see \cite{AD2, Ri}) and dynamic convex risk measures (convex DRMs) (see \cite{FG04, DS, KS, JR, BE, Bi, FS10}), etc. Moreover, several other axiomatic notions in finance are closely related to DRMs. For instance, filtration-consistent nonlinear expectations ($\mathcal{F}$-expectations) introduced by Coquet et al. \cite{CH} have a one-to-one correspondence with a class of DRMs (which we call regular DRMs in the present paper). The dynamic concave utilities introduced by Delbaen et al. \cite{DPR} have a one-to-one correspondence with convex DRMs.

A question raised in \cite{CH} can be generally described in the language of DRMs as follows: for a DRM $\rho$, can we find a function $g$ and a domain ${D}$ such that for all $\xi\in D$,
\begin{equation*}\label{1.1}
 \rho_t(\xi)={\cal{E}}^g[-\xi|{\cal{F}}_t],\quad \forall t\in[0,T]?\tag{1.1}
\end{equation*}
Here, ${\cal{E}}^g[\cdot]$ is the $g$-expectation defined by the solution of the backward stochastic differential equation (BSDE) with generator $g$. This representation problem was considered theoretically very interesting and practically important by Peng in his ICM lecture (see \cite[Page 403]{Peng10}). A positive answer means that any DRM satisfying certain conditions can be computed by solving a corresponding BSDE. This problem has been investigated for regular DRMs (i.e., for $\mathcal{F}$-expectations) in several settings. In \cite{CH}, the authors obtained the first result, which shows that (\ref{1.1}) holds for a Lipschitz continuous function $g$ on the domain $L^2(\mathcal{F}_T)$ under the following domination condition:
\begin{equation*}\label{1.2}
\rho_t(\xi)-\rho_t(\eta)\leq{\cal{E}}^{\mu,0}[\eta-\xi|{\cal{F}}_t],\quad \forall\xi, \eta\in L^2(\mathcal{F}_T),\ \forall t\in[0,T],\tag{1.2}\footnote{From \cite[Lemma 4.3 and Lemma 4.4]{CH} and the fact that ${\cal{E}}^{-\mu,0}[\cdot|{\cal{F}}_t]=-{\cal{E}}^{\mu,0}[-\cdot|{\cal{F}}_t]$, it follows that the assumptions in \cite{CH} (the domination condition: $\rho_0(\xi)-\rho_0(\eta)\leq{\cal{E}}^{\mu,0}[\eta-\xi]$, the translation invariance and the strict monotonicity) imply (\ref{1.2}). In fact, the proof of the representation theorem in \cite{CH} relies essentially on (\ref{1.2}).}
\end{equation*}
where ${\cal{E}}^{\mu,0}$ is the $g$-expectation defined by the BSDE with generator $\mu|z|, \mu\geq0$. This result was later extended to general filtrations under (\ref{1.2}) by Royer \cite{Ro} and Cohen \cite{Co}. To represent regular DRMs more broadly via $g$-expectations, it is natural to extend (\ref{1.2}) to more general cases, especially the quadratic growth case (for example, entropy risk measure is a $g$-expectation whose generator has a quadratic growth (quadratic $g$-expectation)). Zheng and Li \cite{Zheng18} extend the representation theorem in \cite{CH} to the uniformly continuous case by replacing the generator $\mu|z|$ in (\ref{1.2}) with an increasing and linear growth function $\phi(x)$. For the representations of regular DRMs via quadratic $g$-expectations, Hu et al. \cite{HM} first showed that (\ref{1.1}) holds on $L^\infty(\mathcal{F}_T)$ under several domination conditions (see Definition 3.8(1)-(3) and (H4) in \cite{HM}). Later, Zheng \cite{Zheng24} proved that (\ref{1.1}) holds on $L^\infty(\mathcal{F}_T)$ under a locally Lipschitz domination condition together with a convergence assumption (see (A$_{\rho(k)}$) and (A$_{L^\infty}$) in \cite{Zheng24}). Both of these results require that the terminal variables of $\rho$ are bounded and that $\rho$ has independent increments. The latter condition implies that the generator of the $g$-expectation in (\ref{1.1}) is deterministic. However, since BSDEs with superquadratic growth and bounded terminal variables are generally ill-posed (see \cite{DH}), it is generally impossible to represent DRMs on $L^\infty({\cal{F}}_T)$ via superquadratic $g$-expectations. It is noteworthy that all the representation results mentioned above are obtained using the domination argument introduced by \cite{CH}. We refer to \cite{Zheng24} for an alternative method for Problem (\ref{1.1}) in the multi-dimensional case. However, this method requires that $\rho$ has independent increments and satisfies certain convergence properties.

Convex and coherent DRMs are two fundamental classes of risk measures. The main objective of this paper is to study the representation problem (\ref{1.1}) for convex (resp. coherent) DRMs under the following domination condition:
\begin{equation*}\label{1.3}
{\cal{E}}^{-\mu,-\nu}[-\xi|{\cal{F}}_t]\leq\rho_t(\xi)\leq{\cal{E}}^{\mu,\nu}[-\xi|{\cal{F}}_t],\quad \forall\xi\in \mathcal{L}^{\exp}_T\ (\text{resp.}\ L^2(\mathcal{F}_T)),\ \forall t\in[0,T],\tag{1.3}
\end{equation*}
where ${\cal{E}}^{\mu,\nu}$ is the $g$-expectation defined by the BSDE with generator $\mu|z|+\nu|z|^2, \mu\geq0, \nu\geq0$. Since $\rho_t(0)=0$ and ${\cal{E}}^{-\mu,0}[\cdot|{\cal{F}}_t]=-{\cal{E}}^{\mu,0}[-\cdot|{\cal{F}}_t]$, it follows that (\ref{1.2}) implies
${\cal{E}}^{-\mu,0}[-\xi|{\cal{F}}_t]\leq\rho_t(\xi)\leq{\cal{E}}^{\mu,0}[-\xi|{\cal{F}}_t]$. Thus, (\ref{1.2}) is a special case of (\ref{1.3}). For convex (resp. coherent) DRMs, we prove that under (\ref{1.3}), (\ref{1.1}) holds on $\mathcal{L}^{\exp}_T$ (resp. $L^2(\mathcal{F}_T)$) for a generator $g$ which is convex (resp. sublinear) in the variable $z$ and has a quadratic (resp. linear) growth (see Theorem \ref{rmt} and Corollary \ref{rmt2}). Compared to the works of \cite{HM} and \cite{Zheng24}, our result holds for unbounded terminal variables and does not require $\rho$ to have independent increments or satisfy additional convergence properties, and (\ref{1.3}) seems to be easier to verify for convex DRMs in applications. Our proof follows the domination idea developed by \cite{CH} in the Lipschitz case. However, since some results that are crucial in the proof of the representation theorem in \cite{CH} generally fail to hold in non-Lipschitz settings, applying the domination argument in such cases, particularly the quadratic growth case, is really not easy. New domination conditions need to be found, and new strategies and techniques have to be developed, as seen in \cite{HM, Zheng18, Zheng24}. Inspired by the well-posedness results for quadratic BSDEs in Briand and Hu \cite{BH08} and Fan et al. \cite{FHT20}, we use the convexity property to obtain a $\theta$-domination for convex DRMs. This is then used to derive a crucial convergence property of convex DRMs. Combining this convergence, an exponential transformation and localization techniques, we establish a Doob-Meyer decomposition for a $\rho$-supermartingale with the special structure: $Y_t+z\cdot B_t$, where $Y_t$ is bounded. This decomposition, together with the aforementioned convergence result, enables us to complete the proof.

A dual representation of a convex DRM $\rho$ has the form:
\begin{equation*}\label{1.4}
\rho_t(\xi)=\operatorname{ess\,sup}_{Q\in\mathcal{Q}}\left(E_{Q}[-\xi|\mathcal{F}_t]+\alpha(Q)\right),\quad \forall t\in[0,T],\tag{1.4}
\end{equation*}
where $\xi\in L^0(\mathcal{F}_T)$, $\mathcal{Q}$ is a set of probability measures and $\alpha$ is
a penalty function. A dual representation of a coherent DRM $\rho$ has the form:
\begin{equation*}\label{1.5}
\rho_t(\xi)=\operatorname{ess\,sup}_{Q\in\mathcal{Q}_1}E_{Q}[-\xi|\mathcal{F}_t],\quad \forall t\in[0,T],\tag{1.5}
\end{equation*}
where $\xi\in L^0(\mathcal{F}_T)$ and $\mathcal{Q}_1$ is a set of probability measures. For further details on the dual representation of convex DRMs, we refer to Detlefsen and Scandolo \cite{DS}, Bino-Nadal \cite{Bi}, F\"{o}llmer and Schied \cite{FS10}, and the references therein. \cite{DPR} provided a representation for the penalty function $\alpha$ in (\ref{1.4}). They proved that under some additional assumptions, there exists a suitable nonnegative function $f$ such that for all $\xi\in L^\infty(\mathcal{F}_T)$,
\begin{equation*}\label{1.6}
\rho_t(\xi)=\operatorname{ess\,sup}_{Q^q\in\mathcal{Q}}E_{Q^q}
\left[-\xi-\int_t^Tf(s,q_s)ds|\mathcal{F}_t\right],\quad \forall t\in[0,T],\tag{1.6}
\end{equation*}
where $\mathcal{Q}:=\{Q^q: Q^q$ is a probability equivalent to $P$ such that $dQ/dP=\exp\{\int_0^tq_sds-\frac{1}{2}\int_0^t|q_s|^2ds\}\}.$ Using the representation theorem established in this paper (see (\ref{3.2}) in Theorem \ref{rmt}) and some results of Fan et al. \cite{FHT24}, we show that any convex DRM satisfying (\ref{1.3}) admits a dual representation as in (\ref{1.6}). Compared to \cite{DPR}, our result holds for unbounded terminal variables, and the function $f$ can be determined by the Legendre-Fenchel transformation of the generator $g$ in (\ref{3.2}). Furthermore, for coherent DRMs satisfying (\ref{1.3}), we obtain a dual representation as in (\ref{1.5}), where the terminal variable $\xi$ is unbounded and $\mathcal{Q}_1$ is characterized by the subdifferential $\partial g(t,0)$. Our method relies on the representation (\ref{3.2}) and therefore requires (\ref{1.3}).

The paper is organized as follows. In Section 2, we show some results on quadratic $g$-expectations. In Section 3, we introduce the main results of this paper. In Section 4, we prove some properties of convex DRMs. In Section 5, we establish a nonlinear Doob-Meyer decomposition. In Section 6, we provide the proofs of the main results. The Appendix contains an auxiliary proof.
\section{Quadratic $g$-expectations}
Let $(\Omega ,\cal{F},\mathit{P})$ be a complete probability space. Let ${{(B_t)}_{t\geq
0}}$ be a $d$-dimensional standard Brownian motion defined on $(\Omega ,\cal{F},\mathit{P})$. Let $({\cal{F}}_t)_{t\geq 0}$ be
the natural filtration generated by ${{(B_t)}_{t\geq 0}}$ and augmented
by the $\mathit{P}$-null sets of ${\cal{F}}$. Let ${\cal{P}}$ denote the $({\cal{F}}_t)_{t\geq 0}$-progressive sigma-field on $[0,T]\times\Omega$. Let $|z|$ denote the
Euclidean norm of $z\in {\textbf{R}}^d$. For $x,z\in {\textbf{R}}^d$, let $x\cdot z$ denote its scalar product. Let $T>0$ be given real number and ${\cal{T}}_{t,T}$ be the set of all stopping times $\tau$ such that $t\leq \tau\leq T$. Let $\mu\geq0$ and $\nu\geq0$ be given constants. Note that in this paper, all the equalities and inequalities for random variables are understood in the almost sure sense. We define the following spaces:

$L^0({\cal{F}}_t):=\{\xi:\ {\cal{F}}_t$-measurable random variable$\}$, $t\in[0,T]$;

$L^r({\mathcal {F}}_t):=\{\xi\in L^0({\cal{F}}_t):$ ${{E}}\left[|\xi|^r\right]<\infty\},\ r\geq1,\ t\in[0,T]$;

$L^\infty({\mathcal {F}}_t):=\{\xi\in L^0({\cal{F}}_t):$ $\|\xi\|_\infty:=\operatorname{ess\,sup}_{\omega\in\Omega}|\xi(\omega)|<\infty\},\ t\in[0,T]$;

${\mathcal{C}}:=\{(\psi_t)_{t\in[0,T]}:$ continuous and $({\cal{F}}_t)$-adapted process$\}$;

${\mathcal{S}}^r:=\{(\psi_t)_{t\in[0,T]}\in{\mathcal{C}}:$ $E[{\mathrm{sup}}_{0\leq t\leq T}|\psi_t|^r]<\infty \},\ r\geq1;$

${\mathcal{S}}^\infty:=\{(\psi_t)_{t\in[0,T]}\in{\mathcal{C}}:$ $\|\psi_t\|_{\mathcal{S}^\infty}:=\operatorname{ess\,sup}_{(\omega,t)\in\Omega\times[0,T]}|\psi_t(\omega)|<\infty\}$;

$\mathcal{L}^{\exp}_t:=\{\xi: \exp(|\xi|)\in \bigcap_{r>1}L^r({\mathcal {F}}_t)\},\ t\in[0,T]$;

$\mathcal{L}^{\exp}_{\mathcal{F}}:=\{(\psi_t)_{t\in[0,T]}: \exp(|\psi_t|)\in\bigcap_{r>1}\mathcal{S}^r\}$.

${\mathcal{A}}:=\{(\psi_t)_{t\in[0,T]}:$ increasing, RCLL and $({\cal{F}}_t)$-adapted
$\textbf{R}$-valued process with $\psi_0=0\}$;

${H}_d^r:=\{(\psi_t)_{t\in[0,T]}:$  $\textbf{R}^d$-valued, $({\cal{F}}_t)$-progressively measurable and $\int_0^T|\psi_t|^rdt
<\infty \},\ r\geq1;$

${\cal{H}}_d^r:=\{(\psi_t)_{t\in[0,T]}\in{H}^2:$ $E[(\int_0^T|\psi_t|^2dt)^{\frac{r}{2}}]<\infty\},\ r\geq1;$

${\cal{H}}_d^{BMO}:=\{(\psi_t)_{t\in[0,T]}\in{\mathcal{H}}_d^2:$ $\|\psi_t\|_{BMO}:=\sup_{\tau\in{\cal{T}}_{0,T}}\|E[\int_\tau^T|\psi_t|^2dt|{\cal{F}}_\tau]\|_{\infty}^{1/2}<\infty\};$

${\cal{G}}:=\{g(\omega,t,z): \Omega\times [0,T]\times\mathbf{R}^{\mathit{d}}\longmapsto\mathbf{R}$ and is measurable with respect to ${\cal{P}}\otimes{\cal{B}}(\mathbf{R}^d)\};$

${\cal{G}}^{\mu,\nu}:=\{g\in{\cal{G}}:$ $dt\times dP$-$a.e.,$ $\forall z\in\mathbf{R}^{\mathit{d}},\ |g(t,z)|\leq\mu|z|+\nu|z|^2\};$

${\cal{G}}^{\mu,\nu}_{\theta}:=\{g\in{\cal{G}}:$ $dt\times dP$-$a.e.,$ $\forall z,\tilde{z}\in\mathbf{R}^d$, $\forall\theta\in(0,1)$, $g(t,z)-\theta g(t,\tilde{z})\leq\mu|z-\theta\tilde{z}|
+\frac{\nu}{1-\theta}|z-\theta\tilde{z}|^2\}$;

${\cal{G}}_{conv}:=\{g\in{\cal{G}}:$ $dt\times dP$-$a.e.,$ $g(t,z)$ is convex in $z\};$

${\cal{G}}_{subl}:=\{g\in{\cal{G}}:$ $dt\times dP$-$a.e.,$ $\forall z\in\mathbf{R}^d, \forall\beta\geq0$, $g(t,\beta z)=\beta g(t,z)$, and $\forall z_1, z_2\in\mathbf{R}^d, g(t,z_1+z_2)\leq g(t,z_1)+g(t,z_2)\}.$

\begin{remark}\label{r2.1}
For $g\in{\cal{G}}^{\mu,\nu}\cap{\cal{G}}^{\mu,\nu}_{\theta}$, we have $dt\times dP$-$a.e.,$ for all $z, \tilde{z}\in\mathbf{R}^d$ and $\theta\in(0,1)$,
\begin{align*}
&g(t,z)-g(t,\tilde{z})\\&=g(t,z)-\theta g(t,\tilde{z})-(1-\theta)g(t,\tilde{z})\\
  &\leq\mu|z-\theta\tilde{z}|
+\frac{\nu}{1-\theta}|z-\theta\tilde{z}|^2+(1-\theta)(\mu|\tilde{z}|+\nu|\tilde{z}|^2)\\
&\leq\mu(|z-\theta\tilde{z}|+|\tilde{z}-\theta z|)
+\frac{\nu}{1-\theta}(|z-\theta\tilde{z}|^2+|\tilde{z}-\theta z|^2)+(1-\theta)(\mu(|\tilde{z}|+|z|)+\nu(|\tilde{z}|^2+|z|^2)).
\end{align*}
By interchanging $z$ and $\tilde{z}$, we have
\begin{equation*}
\lim_{\tilde{z}\rightarrow z}|g(t,z)-g(t,\tilde{z})|\leq4(1-\theta)(\mu|z|+\nu|z|^2)\rightarrow0,\quad \text{as}\ \theta\rightarrow1.
\end{equation*}
It follows that for all $g\in{\cal{G}}^{\mu,\nu}\cap{\cal{G}}^{\mu,\nu}_{\theta}$, $dt\times dP$-$a.e.,$ $g(t,z)$ is continuous in $z$.
Furthermore, from \cite[Proposition 2.3(i)]{FHT20}, it follows that ${\cal{G}}^{\mu,\nu}\cap{\cal{G}}_{conv}\subset{\cal{G}}^{\mu,\nu}\cap{\cal{G}}^{\mu,\nu}_{\theta}.$
\end{remark}

We consider the following BSDE$(g,\xi,T)$:
\begin{equation*}
Y_t=\xi+\int_t^Tg\left(s,Z_s\right)
ds-\int_t^TZ_s\cdot dB_s,\ \ t\in[0,T].
\end{equation*}
From \cite[Theorem 2.6(i)]{FHT20} and Remark \ref{r2.1}, it is known that for $\xi\in\mathcal{L}^{\exp}_T$ and $g\in{\cal{G}}^{\mu,\nu}\cap{\cal{G}}^{\mu,\nu}_{\theta}$, the BSDE$(g,\xi,T)$ has a unique solution $(Y_t,Z_t)\in\mathcal{L}^{\exp}_{\mathcal{F}}\times\bigcap_{r>1}\mathcal{H}_d^r$. Moreover, by Pardoux and Peng \cite{PP} and \cite[Remark 2.1]{FHT24}, for $\xi\in L^2(\mathcal{F}_T)$ and $g\in{\cal{G}}^{\mu,0}\cap{\cal{G}}_{conv}$, the BSDE$(g,\xi,T)$ has a unique solution $(Y_t,Z_t)\in\mathcal{S}^2\times\mathcal{H}_d^2$. We introduce the notion of $g$-expectations, which is initiated in Peng \cite{Peng97} in the Lipschitz case.

\begin{definition}\label{qg} (i) For $\xi\in \mathcal{L}^{\exp}_T$ and $g\in{\cal{G}}^{\mu,\nu}\cap{\cal{G}}^{\mu,\nu}_{\theta}$, if $(Y_t,Z_t)$ is a solution to the BSDE$(g,\xi,T)$ in $\mathcal{L}^{\exp}_\mathcal{F}\times\bigcap_{r>1}\mathcal{H}_d^r$, then the conditional $g$-expectation of $\xi$ is defined by
${\cal{E}}^g[\xi|{\cal{F}}_t]:=Y_t,\ t\in[0,T]$ and the $g$-expectation of $\xi$ is defined by ${\cal{E}}^g[\xi]:=Y_0.$

(ii) For $\xi\in L^2(\mathcal{F}_T)$ and $g\in{\cal{G}}^{\mu,0}\cap{\cal{G}}_{conv}$, if $(Y_t,Z_t)$ is a solution to the BSDE$(g,\xi,T)$ in $\mathcal{S}^2\times\mathcal{H}_d^2$, then the conditional $g$-expectation of $\xi$ is defined by
${\cal{E}}^g[\xi|{\cal{F}}_t]:=Y_t,\ t\in[0,T]$ and the $g$-expectation of $\xi$ is defined by ${\cal{E}}^g[\xi]:=Y_0.$
\end{definition}

Since ${\cal{G}}^{\mu,\nu}\cap{\cal{G}}_{conv}\subset{\cal{G}}^{\mu,\nu}\cap{\cal{G}}^{\mu,\nu}_{\theta}$ (see Remark \ref{r2.1}), it is clear that the two cases of Definition \ref{qg} are consistent. In the following, we list some properties of quadratic $g$-expectations. For more results on $g$-expectations, we refer to \cite{Peng97, Peng99, Peng04, Ros, Ro, Jiang, MY, Jia, Xu, Zheng24} and the references therein.

\begin{lemma}\label{g} Let $\xi\in\mathcal{L}^{\exp}_T$ and $g,f\in{\cal{G}}^{\mu,\nu}\cap{\cal{G}}^{\mu,\nu}_{\theta}$. Then the following hold:

(i) For all $t\in[0,T]$ and $s\in[0,t]$, ${\cal{E}}^g[{\cal{E}}^g[\xi|{\cal{F}}_t]|{\cal{F}}_s]={\cal{E}}^g[\xi|{\cal{F}}_s]$;

(ii) For all $t\in[0,T]$ and $\zeta\in\mathcal{L}^{\exp}_t$, ${\cal{E}}^g[\xi+\zeta|{\cal{F}}_t]={\cal{E}}^g[\xi|{\cal{F}}_t]+\zeta$;

(iii) For all $t\in[0,T]$ and $\zeta\in\mathcal{L}^{\exp}_t$, ${\cal{E}}^g[\zeta|{\cal{F}}_t]=\zeta$;

(iv) For $t\in[0,T]$, set $\tau_t:=\{s\geq0:|B_{t+s}-B_t|>1\}\wedge (T-t)$. Then, for all $z\in\mathbf{R}^d$,
\begin{equation*}
  \lim_{\epsilon\rightarrow0^+}{\cal{E}}^g[z\cdot (B_{t+\epsilon\wedge\tau_t}-B_t)|{\cal{F}}_t]=g(t,z),\ \ dt\text{-}a.e.;
\end{equation*}

(v) For all $\beta>1$, there exists a constant $C>0$ depending only on $\beta$, $\mu$, $\nu$ and $T$, such that
\begin{equation*}
E\left[\exp\left(2(\mu+\nu)\beta\sup_{t\in[0,T]}|{\cal{E}}^g[\xi|{\cal{F}}_t]|\right)\right]\leq CE\left[\exp\left(2(\mu+\nu)\beta|\xi|\right)\right].
\end{equation*}

(vi) For all $\eta\in\mathcal{L}^{\exp}_T$, if $\xi\leq \eta$, then for all $t\in[0,T]$, ${\cal{E}}^g[\xi|{\cal{F}}_t]\leq{\cal{E}}^g[\eta|{\cal{F}}_t]$;

(vii) $g(t,\cdot)\leq f(t,\cdot)$, $dt\times dP$-a.e., if and only if for all $\eta\in\mathcal{L}^{\exp}_T$,
\begin{equation*}
{\cal{E}}^g[\eta|{\cal{F}}_t]\leq{\cal{E}}^f[\eta|{\cal{F}}_t],\quad \forall t\in[0,T];
\end{equation*}

(viii) $g\in\mathcal{G}_{conv}$ if and only if for all $\eta, \zeta\in\mathcal{L}^{\exp}_T$ and $\theta\in(0,1)$,
\begin{equation*}
{\cal{E}}^g[\theta\eta+(1-\theta)\zeta|{\cal{F}}_t]\leq\theta{\cal{E}}^g[\eta|{\cal{F}}_t]
+(1-\theta){\cal{E}}^g[\zeta|{\cal{F}}_t],\quad \forall t\in[0,T];
\end{equation*}

(ix) $g\in\mathcal{G}_{subl}$ if and only if for all $\eta, \zeta\in\mathcal{L}^{\exp}_T$ and $\beta\geq0$,
\begin{equation*}
{\cal{E}}^g[\beta\eta|{\cal{F}}_t]=\beta{\cal{E}}^g[\eta|{\cal{F}}_t]\ \ \text{and}\ \  {\cal{E}}^g[\eta+\zeta|{\cal{F}}_t]\leq{\cal{E}}^g[\eta|{\cal{F}}_t]
+{\cal{E}}^g[\zeta|{\cal{F}}_t],\quad \forall t\in[0,T].
\end{equation*}
\end{lemma}
\begin{proof}
By the uniqueness of solutions, we obtain (i) and (ii). Since for all $t\in[0,T]$, ${\cal{E}}^g[0|{\cal{F}}_t]=0$, by (ii), we obtain (iii). (iv) is from Zheng \cite[Corollary 3.2]{Zheng15}. Since $|g(t,z)|\leq(\mu+\nu)(1+|z|^2)$, (v) can be derived from \cite[Corollary 4]{BH08}. By the uniqueness of solutions and the comparison theorem in \cite[Theorem 2.7(i)]{Zheng24}, we get (vi). In view of (iv) and the comparison theorem in \cite[Theorem 2.7(ii)]{Zheng24}, by similar arguments as in \cite[Proposition 3.5]{Zheng24}, we can get (vii)-(ix).
\end{proof}

\begin{definition} Let $g\in{\cal{G}}^{\mu,\nu}\cap{\cal{G}}^{\mu,\nu}_{\theta}$. A process $(Y_t)_{t\in[0,T]}$ with $Y_t\in \mathcal{L}^{\exp}_t$ for $t\in [0,T],$ is called a $g$-martingale (resp. $g$-supermartingale, $g$-submartingale), if for any $0\leq s\leq t\leq T,$ we have ${\cal{E}}^g[Y_t|{\cal{F}}_s]=Y_s$, (resp. $\leq,\ \geq$).
\end{definition}

For convenience, we denote $g^{\mu,\nu}(z):=\mu|z|+\nu|z|^2$, $g^{-\mu,-\nu}(z):=-\mu|z|-\nu|z|^2$, $\mathcal{E}^{\mu,\nu}:=\mathcal{E}^{g^{\mu,\nu}}$ and $\mathcal{E}^{-\mu,-\nu}:=\mathcal{E}^{g^{-\mu,-\nu}}$. The path regularity of bounded quadratic $g$-submartingales was obtained by Ma and Yao \cite{MY}. The following is a path regularity result of unbounded $g^{\mu,\nu}$-submartingales.

\begin{lemma}\label{rcll}
Let $Y_t$ be a $g^{\mu,\nu}$-submartingale (resp. $g^{-\mu,-\nu}$-supermartingale). Then $Y_t$ admits an RCLL modification.
\end{lemma}

\begin{proof} Let $Y_t$ be a $g^{\mu,\nu}$-submartingale. By Zheng et al. \cite[Corollary 3.6]{Zheng25}, we have
\begin{equation*}
Y_s\leq \mathcal{E}^{\mu,\nu}[Y_t|{\cal{F}}_s]=\frac{1}{2\nu}\ln(E_Q[\exp(2\nu Y_t)|{\cal{F}}_s]),\ \ \forall 0\leq s\leq t\leq T.
\end{equation*}
Here, $Q$ is a probability measure such that $dQ/dP=\exp\{\int_0^Tb_s\cdot dB_s-\frac{1}{2}\int_0^T|b_s|^2ds\}$, where $b_t$ is a progressively measurable process such that $|b_t|\leq\mu.$ This implies that for any $0\leq s\leq t\leq T$, $\exp(2\nu Y_s)\leq E_Q[\exp(2\nu Y_t)|{\cal{F}}_s]$. Thus, $\exp(2\nu Y_t)$ is a classical submartingale under $Q$. Then by the classical martingale theory, we get that $Y_t$ admits an RCLL modification.

If $Y_t$ is a $g^{-\mu,-\nu}$-supermartingale, then
$-Y_s\leq -{\cal{E}}^{-\mu,-\nu}[Y_t|{\cal{F}}_s]={\cal{E}}^{\mu,\nu}[-Y_t|{\cal{F}}_s]$ for any $0\leq s\leq t\leq T$.
This implies that $-Y_t$ is a $g^{\mu,\nu}$-submartingale. Thus $Y_t$ admits an RCLL modification.
\end{proof}

\begin{remark}\label{rrcll}
In view of Lemma {\ref{rcll}}, for any $g^{\mu,\nu}$-submartingale (resp. $g^{-\mu,-\nu}$-supermartingale), we will always take its RCLL version.
\end{remark}

At the end of this section, we give a Doob-Meyer decomposition of $g^{\mu,\nu}$-submartingales.
\begin{lemma}\label{dmg} Let $Y_t$ be a $g^{\mu,\nu}$-submartingale such that $\sup_{t\in[0,T]}\exp(|Y_t|)\in\bigcap_{r>1}L^r(\mathcal{F}_T)$. Then there exists a pair $(Z_t,A_t)\in H_d^2\times{\cal{A}}$, such that
\begin{equation*}\label{2.1}
Y_t=Y_T+\int_t^T(\mu|Z_s|+\nu|Z_s|^2)
ds-A_T+A_t-\int_t^TZ_s\cdot dB_s,\ \  t\in[0,T].\tag{2.1}
\end{equation*}
\end{lemma}
\begin{proof} For $t\in[0,T]$, applying It\^{o}'s formula to $\exp(2\nu \mathcal{E}^{\mu,\nu}[Y_t|{\cal{F}}_s])$ for $s\in[0,t]$, we have
\begin{equation*}\label{2.2}
\exp(2\nu\mathcal{E}^{\mu,\nu}[Y_t|{\cal{F}}_s])=\mathcal{E}^{\mu,0}[\exp(2\nu Y_t)|{\cal{F}}_s],\quad\forall s\in[0,t].\tag{2.2}
\end{equation*}
Since for all $s\in[0,t]$, $\exp(2\nu Y_s)\leq\exp(2\nu\mathcal{E}^{\mu,\nu}[Y_t|{\cal{F}}_s])$, by (\ref{2.2}), $\exp(2\nu Y_t)$ is a $g^{\mu,0}$-submartingale. Then, by Peng \cite[Theorem 3.3]{Peng99}, there exists a pair $(z_t, a_t)\in H_d^2\times{\cal{A}}$, such that
\begin{equation*}
\exp(2\nu Y_t)=\exp(2\nu Y_T)+\int_t^T\mu|z_s|
ds-a_T+a_t-\int_t^Tz_s\cdot dB_s,\quad t\in[0,T].\\
\end{equation*}
Let $a^c_t$ be the continuous part of $a_t$. Set $y_t:=\exp(2\nu Y_t),\ \Delta Y_t:=Y_t-Y_{t-},\ \Delta y_t:=y_t-y_{t-},\ \Delta a_t:=a_t-a_{t-},\ t\in(0,T]$. Applying It\^{o}'s formula to $\frac{1}{2\nu}\ln(y_t)$, we have
\begin{align*}
Y_t=Y_T+&\int_{t+}^T\frac{1}{2\nu y_{s-}}\left(\mu|z_s|ds-da_s-z_s\cdot dB_s\right)
+\int_{t+}^T\frac{|z_s|^2}{4\nu |y_s|^2}ds-\sum_{s\in(t,T]}\left(\Delta Y_s-\frac{1}{2\nu y_{s-}}\Delta y_s\right)\\
=Y_T+&\int_t^T\left(\mu\left|\frac{z_s}{2\nu y_s}\right|+\nu\left|\frac{z_s}{2\nu y_s}\right|^2\right)
ds-\int_t^T\frac{z_s}{2\nu y_s}\cdot dB_s\\
&-\int_t^T\frac{1}{2\nu y_s}da^c_s-\sum_{s\in(t,T]}\left(\frac{1}{2\nu y_{s-}}\Delta a_s
+\Delta Y_s-\frac{1}{2\nu y_{s-}}\Delta y_s\right),\quad t\in[0,T].
\end{align*}
Since $\Delta y_t=\Delta a_t\geq0$, by setting $Z_t:=\frac{z_s}{2\nu y_s}$ and $A_t:=\int_0^t\frac{1}{2\nu y_s}da^c_s+\sum_{s\in(0,t]}\Delta Y_s$, we get that $(Z_t,A_t)\in H_d^2\times{\cal{A}}$ and (\ref{2.1}) holds. The proof is complete.
\end{proof}
\section{Representations of dynamic convex risk measures}
We first introduce some notions of dynamic risk measures (DRMs).
\begin{definition}\label{defrm}
A system of operators:
\begin{equation*}
\rho_t(\cdot):\ L^0(\mathcal{F}_T)\longrightarrow L^0(\mathcal{F}_t), \quad t\in[0,T],
\end{equation*}
is called a DRM on $L^0(\mathcal{F}_T)$, if for all $\xi, \eta\in L^0(\mathcal{F}_T)$, it satisfies (\hyperref[r1]{r1}), (\hyperref[r2]{r2}) and (\hyperref[r3]{r3}):

(r1)\label{r1} Monotonicity: If $\xi\geq\eta$, then for all $t\in[0,T]$, $\rho_t(\xi)\leq\rho_t(\eta)$;

(r2)\label{r2} Time-consistency: For all $t\in[0,T]$ and $s\in[0,t]$, $\rho_s(-\rho_t(\xi))=\rho_s(\xi)$;

(r3)\label{r3} Constant preservation: If for some $t\in[0,T]$, $\xi\in L^0(\mathcal{F}_t),$ then $\rho_t(\xi)=-\xi$.\\
A DRM is called a \textbf{convex} DRM, if for all $\xi, \eta\in L^0(\mathcal{F}_T)$, it satisfies (\hyperref[r4]{r4}):

(r4)\label{r4} Convexity: For all $t\in[0,T]$ and $\theta\in(0,1)$, $\rho_t(\theta\xi+(1-\theta)\eta)\leq\theta\rho_t(\xi)+(1-\theta)\rho_t(\eta).$\\
A DRM is called a \textbf{coherent} DRM, if for all $\xi, \eta\in L^0(\mathcal{F}_T)$, it satisfies (\hyperref[r5]{r5}) and (\hyperref[r6]{r6}):

(r5)\label{r5} Subadditivity: For all $t\in[0,T]$, $\rho_t(\xi+\eta)\leq\rho_t(\xi)+\rho_t(\eta)$;

(r6)\label{r6} Positive homogeneity: For all $\lambda\geq0$ and $t\in[0,T]$, $\rho_t(\lambda\xi)=\lambda\rho_t(\xi)$.\\
A DRM is called a \textbf{regular} DRM, if for all $\xi\in L^0(\mathcal{F}_T)$, it satisfies (\hyperref[r7]{r7}):

(r7)\label{r7} Regular: For all $t\in[0,T]$ and $A\in{\mathcal {F}}_t,$
$\rho_t(1_A\xi)=1_A\rho_t(\xi)$.
\end{definition}

It can be seen that a coherent DRM is a convex DRM. From Lemma \ref{g}(viii), it is clear that if $g\in{\cal{G}}^{\mu,\nu}\cap{\cal{G}}_{conv}$, then ${\cal{E}}^g[-\cdot|{\cal{F}}_t]$ is a convex DRM on $\mathcal{L}^{\exp}_T$; From Rosazza Gianin \cite{Ro}, it follows that if $g\in{\cal{G}}^{\mu,0}\cap{\cal{G}}_{conv}$, then ${\cal{E}}^g[-\cdot|{\cal{F}}_t]$ is a convex DRM on $L^2(\mathcal{F}_T)$.

\begin{remark}\label{rrm} We present the correspondences between DRMs and several important notions in financial mathematics:

(i) If $\rho$ is a coherent DRM, then $\rho_0$ is a coherent risk measure (see \cite{AD});

(ii) If $\rho$ is a convex DRM, then $\rho_0$ is a convex risk measure (see \cite{FS, FG});

(iii) $\rho$ is a regular DRM if and only if $\rho(-\cdot)$ is an $\mathcal{F}$-expectation (see \cite{CH, Peng04});

(iv) $\rho$ is a convex DRM if and only if $-\rho$ is a dynamic concave utility (see \cite{DPR}).
\end{remark}

To solve Problem (\ref{1.1}) for convex DRMs, we make the following assumption:
\begin{itemize}
  \item \textbf{(A)}\label{A}\quad For all $\xi\in\mathcal{L}^{\exp}_T$,
\begin{equation*}\label{3.1}
  {\cal{E}}^{-\mu,-\nu}[-\xi|{\cal{F}}_t]\leq\rho_t(\xi)\leq{\cal{E}}^{\mu,\nu}[-\xi|{\cal{F}}_t],\quad \forall t\in[0,T].
  \tag{3.1}
\end{equation*}
\end{itemize}

We have the following representation theorems:

\begin{theorem}\label{rmt} Let $\rho$ be a convex DRM satisfying (\hyperref[A]{A}). Then there exists a unique $g\in{\cal{G}}^{\mu,\nu}\cap{\cal{G}}_{conv}$ such that for all $\xi\in \mathcal{L}^{\exp}_T$,
\begin{equation*}\label{3.2}
 \rho_t(\xi)={\cal{E}}^g[-\xi|{\cal{F}}_t],\quad \forall t\in[0,T].\tag{3.2}
\end{equation*}
Moreover, the function $f(t,x):=\sup_{z\in\mathbf{R}^d}\{x\cdot z-g(t,z)\},\ (t,x)\in[0,T]\times\mathbf{R}^d,$
is such that for all $\xi\in\mathcal{L}^{\exp}_T$,
\begin{equation*}\label{3.3}
\rho_t(\xi)=\operatorname{ess\,sup}_{q\in\mathcal{Q}_{\xi,f}}E_{Q^q}\left[-\xi-\int_t^Tf(s,q_s)ds|\mathcal{F}_t\right],\quad \forall t\in[0,T],\tag{3.3}
\end{equation*}
where $\mathcal{Q}_{\xi,f}:=\{q_t\in H_d^2: \theta^q_t:=\exp\{\int_0^tq_s\cdot dB_s-\frac{1}{2}\int_0^t|q_s|^2ds\}$ is a uniformly integrable martingale such that $E_{Q^q}[|\xi|+\int_0^T|f(s,q_s)|ds]<\infty$ with $dQ^q=\theta^q_TdP\}.$ In particular, the supremum in (\ref{3.3}) is attained by any $q_t\in H_d^2$ such that $dt\times dP$-$a.e.,$ $q_t\in \partial g(t,Z_t^{-\xi,g})$.\footnote{Here, $Z^{-\xi,g}_t$ is the second component of the solution to the BSDE$(g,-\xi,T)$ in $\mathcal{L}^{\exp}_\mathcal{F}\times\bigcap_{r>1}\mathcal{H}_d^r$.}
\end{theorem}

\begin{cor}\label{rmt2} Let $\rho$ be a coherent DRM satisfying (\hyperref[A]{A}). Then there exists a unique $g\in{\cal{G}}^{\mu,0}\cap{\cal{G}}_{subl}$ such that for all $\xi\in \mathcal{L}^{\exp}_T$, (\ref{3.2}) holds, and $g$ is also a unique function in ${\cal{G}}^{\mu,0}\cap{\cal{G}}_{subl}$ such that for all $\xi\in \mathcal{L}^{\exp}_T$,
     \begin{equation*}\label{3.4}
     \rho_t(\xi)=\operatorname{ess\,sup}_{Q\in\mathcal{Q}_g}E_{Q}\left[-\xi|\mathcal{F}_t\right],\quad \forall t\in[0,T],\tag{3.4}
\end{equation*} where $\mathcal{Q}_g:=\{Q^q:$ $\frac{dQ^q}{dP}=\exp\{\int_0^Tq_s\cdot dB_s-\frac{1}{2}\int_0^T|q_s|^2ds\}$ with $q_t\in H^2_d$ and $dt\times dP$-$a.e.,$ $q_t\in\partial g(t,0)\}$.
Moreover, if for all $\xi\in L^2(\mathcal{F}_T)$, $\rho$ satisfies (\ref{3.1}) with $\nu=0$, then the function $g$  is such that for all $\xi\in L^2(\mathcal{F}_T)$, both (\ref{3.2}) and (\ref{3.4}) hold.
\end{cor}

The proofs of Theorem \ref{rmt} and Corollary \ref{rmt2} will be given in Section \ref{sec6}. In the following, we make some remarks on Theorem \ref{rmt} and Corollary \ref{rmt2}.

\begin{remark}\label{rrmt1}
\begin{itemize}
\item[(i)] For representation theorems via quadratic $g$-expectations, \cite[Theorem 6.3]{HM} proved that if the regular DRM $\rho$ satisfies certain domination conditions (see Definition 3.8(1)-(3) and (H4) in \cite{HM}), then (\ref{3.2}) holds on $L^\infty({\cal{F}}_T)$. \cite[Theorem 4.6]{Zheng24} showed that if the regular DRM $\rho$ satisfies a locally Lipschitz domination condition together with a convergence assumption (see (A$_{\rho(k)}$) and (A$_{L^\infty}$) in \cite{Zheng24}), then (\ref{3.2}) holds on $L^\infty({\cal{F}}_T)$. Both results require $\rho$ to have independent increments, which implies that the generator $g$ in (\ref{3.2}) is deterministic. In contrast, in Theorem \ref{rmt}, the terminal variables of $\rho$ are unbounded and we do not require that $\rho$ has independent increments or satisfies additional convergence properties. Moreover, for convex DRMs, (\hyperref[A]{A}) seems to be easier to verify in applications than the corresponding conditions in \cite{HM} and \cite{Zheng24}. Some remarks on our proof can be found in the Introduction.
\item[(ii)] \cite[Theorem 3.2]{DPR} proved that any convex DRM admits a dual representation of the form (\ref{3.3}) on $L^\infty({\cal{F}}_T)$ under certain additional assumptions. Our Theorem \ref{rmt}, in contrast, provides a representation for unbounded terminal variables, where the function $f$ is explicitly given by the Legendre-Fenchel transformation of the generator $g$ in (\ref{3.2}).  Corollary \ref{rmt2} provides a dual representation for coherent DRMs in which the set of probability measures is characterized via the subdifferential $\partial g(t,0)$, a distinctive feature compared to those commonly found in the literature. Our method relies on (\ref{3.2}) and therefore requires (\hyperref[A]{A}).
\end{itemize}
\end{remark}

The following is a typical example of convex DRMs.

\begin{example} (Dynamic entropy risk measure) For $\xi\in\mathcal{L}^{\exp}_T$,
\begin{equation*}
 \rho_t^{\mathrm{ent}}(\xi):=\frac{1}{2\nu}\ln E[\exp(-2\nu\xi)|\mathcal{F}_t],\quad t\in[0,T],
\end{equation*}
is a convex DRM, called dynamic entropy risk measure. By It\^{o}'s formula, we have
\begin{equation*}
\rho^{\mathrm{ent}}_t(\xi)=\mathcal{E}^{0,\nu}[-\xi|\mathcal{F}_t],\quad t\in[0,T],
\end{equation*}
and thus $\rho$ satisfies (\hyperref[A]{A}). From Theorem \ref{rmt}, it follows that
\begin{equation*}\label{3.5}
\rho^{\mathrm{ent}}_t(\xi)=\operatorname{ess\,sup}_{q\in\mathcal{Q}_{\xi,f}}E_{Q^q}
\left[-\xi-\int_t^T\frac{|q_s|^2}{4\nu}ds|\mathcal{F}_t\right],\quad \forall t\in[0,T].\tag{3.5}
\end{equation*}
In particular, when $\xi\in L^\infty(\mathcal{F}_T)$, by (\ref{a.4}) in the \hyperref[appA]{Appendix}, $\mathcal{Q}_{\xi,f}$ in (\ref{3.4}) can be reduced to $\mathcal{Q}:=\{Q^q:$ $dQ^q/dP=\theta^q_T:=\int_0^Tq_s\cdot dB_s-\frac{1}{2}\int_0^T|q_s|^2ds$ with $q_t\in\mathcal{H}_d^{BMO}\}.$ For $Q^q\in\mathcal{Q}$, the conditional relative entropy of $Q^q$ with respect to $P$:
\begin{align*}\label{3.6}
\mathbb{H}_t(Q^q|P):=E_{Q^q}\left[\ln\left(\frac{\theta^q_T}{\theta^q_t}\right)|\mathcal{F}_t\right]
&=E_{Q^q}\left[\left(\int_t^Tq_s\cdot dB_s-\frac{1}{2}\int_t^T|q_s|^2ds\right)|\mathcal{F}_t\right]\\
&=E_{Q^q}\left[\left(\int_t^Tq_s\cdot dB^{q}_s+\frac{1}{2}\int_t^T|q_s|^2ds\right)|\mathcal{F}_t\right]\\
&=\frac{1}{2}E_{Q^q}\left[\int_t^T|q_s|^2ds|\mathcal{F}_t\right],\quad \forall t\in[0,T],\tag{3.6}
\end{align*}
where $B^{q}_t:=B_t-\int_0^t{q}_sds$ is a standard Brownian motion under the probability $Q^{q}$. From (\ref{3.5}) with $\mathcal{Q}_{\xi,f}$ replaced by $\mathcal{Q}$, and (\ref{3.6}), it follows that for all $\xi\in L^\infty(\mathcal{F}_T)$,
\begin{equation*}\label{3.7}
\rho^{\mathrm{ent}}_t(\xi)=\operatorname{ess\,sup}_{Q^q\in\mathcal{Q}}\left(E_{Q^q}[-\xi|\mathcal{F}_t]
-\frac{1}{2\nu}\mathbb{H}_t(Q^q|P)\right),\quad \forall t\in[0,T].\tag{3.7}
\end{equation*}
A result similar to (\ref{3.7}) can be found in \cite{FS10}.
\end{example}
\section{Some properties of convex DRMs}
In this section, we show some properties of convex DRMs.
\begin{prop}\label{rmp1}
Let $\rho$ be a convex DRM. Then the following hold:

(i) Translation invariance: For all $\xi\in L^0({\cal{F}}_T)$, $t\in[0,T]$ and $\zeta\in L^0({\cal{F}}_t)$, we have
\begin{equation*}
\rho_t(\xi+\zeta)=\rho_t(\xi)-\zeta;
\end{equation*}

(ii) $\theta$-domination: For all $\xi, \eta\in L^0(\mathcal{F}_T)$ and $\theta\in(0,1)$, we have
\begin{align*}
  \rho_t(\xi)-\theta\rho_t(\eta)
  \leq(1-\theta)\rho_t\left(\frac{\xi-\theta\eta}{1-\theta}\right),\quad \forall t\in[0,T];
\end{align*}

(iii) $L^\infty$-domination: For all $\xi,\eta\in L^\infty({\cal{F}}_T)$ and $z\in \mathbf{R}^d$, we have
\begin{equation*}
\|\rho_t(\xi-z\cdot B_T)-\rho_t(\eta-z\cdot B_T)\|_{\mathcal{S}^\infty}\leq \|\xi-\eta\|_\infty.
\end{equation*}
\end{prop}
\begin{proof}
(i) By (\hyperref[r4]{r4}) and (\hyperref[r3]{r3}), for all $\xi\in L^0({\cal{F}}_T)$, $t\in[0,T]$ and $\zeta\in L^0({\cal{F}}_t)$, we have
\begin{equation*}
\rho_t(\xi+\zeta)\leq\theta\rho_t\left(\frac{\xi}{\theta}\right)
+(1-\theta)\rho_t\left(\frac{\zeta}{1-\theta}\right)
=\theta\rho_t\left(\frac{\xi}{\theta}\right)-\zeta,\quad \forall\theta\in(0,1),
\end{equation*}
and
\begin{equation*}
\rho_t(\xi)=\rho_t(\xi+\zeta-\zeta)
\leq\theta\rho_t\left(\frac{\xi+\zeta}{\theta}\right)
+(1-\theta)\rho_t\left(\frac{-\zeta}{1-\theta}\right)
=\theta\rho_t\left(\frac{\xi+\zeta}{\theta}\right)
+\zeta,\quad \forall\theta\in(0,1).
\end{equation*}
By (\hyperref[r4]{r4}), the function $h(x):=\rho_t(x\xi)$ is convex on $\mathbf{R}$, and thus continuous on $\mathbf{R}.$ Then, by letting $\theta\rightarrow1$ in the two inequalities above, we obtain (i).

(ii) By (\hyperref[r4]{r4}), for all $\xi, \eta\in L^0(\mathcal{F}_T)$ and $\theta\in(0,1)$, we have
\begin{align*}
  \rho_t(\xi)-\theta\rho_t(\eta)
  &=\rho_t\left(\theta\eta+(1-\theta)\frac{\xi-\theta\eta}{1-\theta}\right)
  -\theta\rho_t(\eta)\\
  &\leq(1-\theta)\rho_t\left(\frac{\xi-\theta\eta}{1-\theta}\right),\quad \forall t\in[0,T].
\end{align*}

(iii) By (\hyperref[r1]{r1}) and (i), for all $\xi,\eta\in L^\infty({\cal{F}}_T)$, $t\in[0,T]$ and $z\in\mathbf{R}^d$, we have
\begin{equation*}
\rho_t(\xi-z\cdot B_T)\leq\rho_t(\eta-z\cdot B_T-\|\xi-\eta\|_\infty)=\rho_t(\eta-z\cdot B_T)+\|\xi-\eta\|_\infty.
\end{equation*}
By interchanging $\xi$ and $\eta$, we obtain (iii).
\end{proof}

\begin{prop}\label{rmp2}
Let $\rho$ be a convex DRM satisfying (\hyperref[A]{A}). Then the following hold:

(i) Convergence: For $\xi,\xi_n\in \mathcal{L}^{\exp}_T$ such that for all $\beta>1$, $\sup_{n\geq1}E[\exp(\beta|\xi_n|))]<\infty$, if $\xi_n\rightarrow\xi,$ $dP$-$a.s.$, as $n\rightarrow\infty,$  then for all $\beta>1$,
\begin{equation*}
  \lim_{n\rightarrow\infty}E\left[\exp\left(\sup_{t\in[0,T]}\beta|\rho_t(\xi_n)
  -\rho_t(\xi)|\right)\right]=1;
\end{equation*}

(ii) Regular: For all $\xi\in\mathcal{L}^{\exp}_T$, $t\in[0,T]$ and $A\in{\mathcal {F}}_t$, we have
\begin{equation*}
\rho_t(1_A\xi)=1_A\rho_t(\xi);
\end{equation*}

(iii) Boundedness: For all $\xi\in L^\infty({\cal{F}}_T)$ and $z\in \mathbf{R}^d$, we have
\begin{equation*}
\|\rho_t(\xi-z\cdot B_T)-z\cdot B_t\|_{\mathcal{S}^\infty}<\infty.
\end{equation*}
\end{prop}

\begin{proof}
(i) Since for all $\beta>1$, $\sup_{n\geq1}E[\exp(\beta|\xi_n|))]<\infty$, by Lemma \ref{g}(v), we deduce that for all $\alpha>1$ and $\beta>1$,
\begin{equation*}\label{4.1}
\sup_{n\geq1}E\left[\exp\left(\beta\sup_{t\in{[0,T]}}{\cal{E}}^{\mu,\nu}
[\alpha|\xi_n||{\cal{F}}_t]\right)\right]<\infty. \tag{4.1}
\end{equation*}
Since $\xi_n\rightarrow\xi$, $dP$-$a.s.$, as $n\rightarrow\infty,$ by Fatou's Lemma, we get that for all $\beta>1$,
\begin{equation*}\label{4.2}
E[\exp(\beta|\xi|))]\leq\liminf_{n\rightarrow\infty}E[\exp(\beta|\xi_n|))]
\leq\sup_{n\geq1}E[\exp(\beta|\xi_n|))]<\infty,\tag{4.2}
\end{equation*}
which together with Lemma \ref{g}(v) implies that
for all $\alpha>1$ and $\beta>1$,
\begin{equation*}\label{4.3}
E\left[\exp\left(\beta\sup_{t\in{[0,T]}}{\cal{E}}^{\mu,\nu}
[\alpha|\xi||{\cal{F}}_t]\right)\right]<\infty. \tag{4.3}
\end{equation*}
By Lemma \ref{g}(v), (\ref{4.2}) and dominated convergence theorem, we get that for all $\alpha>1$ and $\beta>1$, there exists a constant $C>0$ depending only on $\beta$, $\mu$, $\nu$ and $T$, such that
\begin{align*}\label{4.4}
\lim_{n\rightarrow\infty}E\left[\exp\left(\beta\sup_{t\in{[0,T]}}{\cal{E}}^{\mu,\nu}
  \left[\alpha|\xi_n-\xi||{\cal{F}}_t\right]\right)\right]&\leq \lim_{n\rightarrow\infty}CE\left[\exp\left(\frac{\beta }{(\mu+\nu)\wedge1}2(\mu+\nu)\alpha |\xi_n-\xi|\right)\right]\\
  &=C.\tag{4.4}
\end{align*}
By Lemma \ref{g}(vi) and (\hyperref[A]{A}), we have
\begin{align*}\label{4.5}
-{\cal{E}}^{\mu,\nu}[|\xi||{\cal{F}}_t]={\cal{E}}^{-\mu,-\nu}[-|\xi||{\cal{F}}_t]\leq{\cal{E}}^{-\mu,-\nu}[-\xi|{\cal{F}}_t]
\leq\rho_t(\xi)
\leq{\cal{E}}^{\mu,\nu}[-\xi|{\cal{F}}_t]\leq{\cal{E}}^{\mu,\nu}[|\xi||{\cal{F}}_t].\tag{4.5}
\end{align*}
By Proposition \ref{rmp1}(ii), (\ref{4.5}) and Lemma \ref{g}(viii)(vi), for all $n\geq1$, $t\in[0,T]$, and $\theta\in(0,1)$, we have
\begin{align*}
  \rho_t(\xi_n)-\rho_t(\xi)&\leq\rho_t(\xi_n)-\theta\rho_t(\xi)-(1-\theta)\rho_t(\xi)\\
  &\leq(1-\theta)\rho_t\left(\frac{\xi_n-\theta\xi}{1-\theta}\right)
  +(1-\theta)|\rho_t(\xi)|\\
  &\leq(1-\theta){\cal{E}}^{\mu,\nu}\left[|\xi_n|+\frac{\theta}{1-\theta}|\xi_n-\xi||{\cal{F}}_t\right]
  +(1-\theta){\cal{E}}^{\mu,\nu}[|\xi||{\cal{F}}_t]\\
  &\leq\frac{1-\theta}{2}\left({\cal{E}}^{\mu,\nu}[2|\xi_n||{\cal{F}}_t]
  +{\cal{E}}^{\mu,\nu}\left[\frac{2\theta}{1-\theta}|\xi_n-\xi||{\cal{F}}_t\right]\right)
  +(1-\theta){\cal{E}}^{\mu,\nu}[|\xi||{\cal{F}}_t]\\
  &\leq(1-\theta)\left({\cal{E}}^{\mu,\nu}[2|\xi_n||{\cal{F}}_t]+{\cal{E}}^{\mu,\nu}[2|\xi||{\cal{F}}_t]
  +{\cal{E}}^{\mu,\nu}\left[\frac{2\theta}{1-\theta}|\xi_n-\xi||{\cal{F}}_t\right]\right).
\end{align*}
By interchanging $\xi_n$ and $\xi$, we have
\begin{align*}
|\rho_t(\xi_n)-\rho_t(\xi)|\leq(1-\theta)\left({\cal{E}}^{\mu,\nu}[2|\xi_n||{\cal{F}}_t]+{\cal{E}}^{\mu,\nu}[2|\xi||{\cal{F}}_t]
  +{\cal{E}}^{\mu,\nu}\left[\frac{2\theta}{1-\theta}|\xi_n-\xi||{\cal{F}}_t\right]\right).
\end{align*}
Then, by Jensen's inequality and H\"{o}lder's inequality, we get that for all $n\geq1$ and $\theta\in(0,1)$,
\begin{align*}
&E\left[\exp\left(\sup_{t\in{[0,T]}}\beta|\rho_t(\xi_n)-\rho_t(\xi)|\right)\right]\\
&\leq E\left[\exp\left((1-\theta)\sup_{t\in{[0,T]}}\beta\left({\cal{E}}^{\mu,\nu}[2|\xi_n||{\cal{F}}_t]
+{\cal{E}}^{\mu,\nu}[2|\xi||{\cal{F}}_t]
  +{\cal{E}}^{\mu,\nu}\left[\frac{2\theta}{1-\theta}|\xi_n-\xi||{\cal{F}}_t\right]\right)\right)\right]\\
&\leq\left[E\exp\left(\sup_{t\in{[0,T]}}\beta\left({\cal{E}}^{\mu,\nu}[2|\xi_n||{\cal{F}}_t]
+{\cal{E}}^{\mu,\nu}[2|\xi||{\cal{F}}_t]
  +{\cal{E}}^{\mu,\nu}\left[\frac{2\theta}{1-\theta}|\xi_n-\xi||{\cal{F}}_t\right]\right)\right)\right]^{(1-\theta)}\\
&\leq\left[\sup_{n\geq1}E\left[\exp\left(4\beta\sup_{t\in{[0,T]}}
{\cal{E}}^{\mu,\nu}[2|\xi_n||{\cal{F}}_t]\right)\right]\right]^{\frac{(1-\theta)}{4}}
\times\left[E\left[\exp\left(4\beta\sup_{t\in{[0,T]}}
{\cal{E}}^{\mu,\nu}[2|\xi||{\cal{F}}_t]\right)\right]\right]^{\frac{(1-\theta)}{4}}\\
  &\ \ \ \ \times\left[E\left[\exp\left(2\beta\sup_{t\in{[0,T]}}{\cal{E}}^{\mu,\nu}
  \left[\frac{2\theta}{1-\theta}|\xi_n-\xi||{\cal{F}}_t\right]\right)\right]\right]^{\frac{(1-\theta)}{2}}
\end{align*}
In view of (\ref{4.1}), (\ref{4.3}) and (\ref{4.4}), letting first $n\rightarrow\infty$ and then $\theta\rightarrow1$, we obtain (i).

(ii) Set $\bar{\xi}_n:=1_{\{|\xi|\leq n\}}\xi,\ n\geq1.$ As showed in Kl\"{o}ppel and Schweizer \cite[Page 602]{KS}, for all $t\in[0,T]$ and $A\in{\mathcal {F}}_t$, we have $1_A\rho_t(1_A\bar{\xi}_n)=1_A\rho_t(\bar{\xi}_n)$, which with the fact that $\rho_t(0)=0$, implies
\begin{equation*}\label{4.6}
\rho_t(1_A\bar{\xi}_n)=1_A\rho_t(1_A\bar{\xi}_n)+1_{A^c}\rho_t(1_A\bar{\xi}_n)=1_A\rho_t(\bar{\xi}_n)+1_{A^c}\rho_t(0)=1_A\rho_t(\bar{\xi}_n).\tag{4.6}
\end{equation*}
By (i), we deduce that for all $p>1$ and $A\in{\mathcal {F}}_t$, $\lim_{n\rightarrow\infty}E[\sup_{t\in[0,T]}|\rho_t(1_A\bar{\xi}_n)-\rho_t(1_A\xi)|^p]=0.$ From this and (4.6), we get (ii).

(iii) For $\xi\in L^\infty({\cal{F}}_T)$ and $z\in \textbf{R}^d$, we consider the BSDE$(g^{\mu,\nu}, -\xi+z\cdot B_T, T):$
\begin{equation*}
  {\cal{E}}^{\mu,\nu}[-\xi+z\cdot B_T|{\cal{F}}_t]=-\xi+z\cdot B_T+\int_t^Tg^{\mu,\nu}(Z_s)ds-\int_t^TZ_s\cdot dB_s, \quad t\in[0,T].
\end{equation*}
Then, we have
\begin{equation*}\label{4.7}
  {\cal{E}}^{\mu,\nu}[-\xi+z\cdot B_T|{\cal{F}}_t]-z\cdot B_t=-\xi+\int_t^Tg^{\mu,\nu}(Z_s)ds-\int_t^T(Z_s-z)\cdot dB_s, \quad t\in[0,T],\tag{4.7}
\end{equation*}
which implies that $({\cal{E}}^{\mu,\nu}[-\xi+z\cdot B_T|{\cal{F}}_t]-z\cdot B_t,Z_t-z)\in\mathcal{L}^{\exp}_\mathcal{F}\times\bigcap_{r>1}\mathcal{H}_d^r$ is a solution to the BSDE$(g^{\mu,\nu}(Z_s),-\xi,T)$. Let $(y^1_t,z_t^1)\in{\cal{S}}^\infty\times H_d^2$ be a solution to the BSDE$(g^{\mu,\nu}(|\cdot|+|z|),-\xi,T)$. Since $dt\times dP$-$a.e.,$
\begin{equation*}
0\leq g^{\mu,\nu}(Z_t)\leq g^{\mu,\nu}(|Z_t-z|+|z|),
\end{equation*}
by (\ref{4.7}) and the comparison theorem (see \cite[Theorem 5]{BH08} and its proof), we get
\begin{equation*}
E[-\xi|{\cal{F}}_t]\leq{\cal{E}}^{\mu,\nu}[-\xi+z\cdot B_T|{\cal{F}}_t]-z\cdot B_t\leq y^1_t,\quad \forall t\in[0,T].
\end{equation*}
which implies that ${\cal{E}}^{\mu,\nu}[-\xi+z\cdot B_T|{\cal{F}}_t]-z\cdot B_t\in{\cal{S}}^\infty.$ This also gives
\begin{equation*}
{\cal{E}}^{-\mu,-\nu}[-\xi+z\cdot B_T|{\cal{F}}_t]-z\cdot B_t
=-({\cal{E}}^{\mu,\nu}[\xi-z\cdot B_T|{\cal{F}}_t]+z\cdot B_t)\in{\cal{S}}^\infty.
\end{equation*}
Then, by (\hyperref[A]{A}), we obtain (iii).
\end{proof}

The following shows that under (\hyperref[A]{A}), for all $\xi\in\mathcal{L}^{\exp}_T,$ $\rho_t(\xi)$ is an It\^{o} process and thus all the paths of $\rho_t(\xi)$ are continuous.
\begin{prop}\label{rp} Let $\rho$ be a convex DRM satisfying (\hyperref[A]{A}). Then for all $\xi\in\mathcal{L}^{\exp}_T,$ there exists a pair $(g_t^\xi, Z_t^\xi)\in H_1^1\times H_d^2$ such that
\begin{equation*}\label{4.8}
\rho_t(\xi)=-\xi+\int_t^Tg_s^\xi ds-\int_t^TZ_s^\xi \cdot dB_s,\quad t\in[0,T],\tag{4.8}
\end{equation*}
with $|g_t^\xi|\leq g^{\mu,\nu}(Z_t^\xi)$, $dt\times dP$-$a.e.$ Moreover, for all $\xi,\eta\in\mathcal{L}^{\exp}_T$ and $\theta\in(0,1)$, we have $dt\times dP$-$a.e.,$
\begin{equation*}
\frac{g_t^\xi-\theta g_t^\eta}{1-\theta}\leq g^{\mu,\nu}\left(\frac{Z_t^\xi-\theta Z_t^\eta}{1-\theta}\right).
\end{equation*}
\end{prop}

\begin{proof} For all $\xi\in\mathcal{L}^{\exp}_T$ and $0\leq s\leq t\leq T$, by (\hyperref[A]{A}), we have $\rho_s(\xi)=\rho_s(-\rho_t(\xi))\leq {\cal{E}}^{\mu,\nu}[\rho_t(\xi)|{\cal{F}}_s]$ and
\begin{equation*}
-\rho_s(\xi)=-\rho_s(-\rho_t(\xi))\leq-{\cal{E}}^{-\mu,-\nu}[\rho_t(\xi)|{\cal{F}}_s]
={\cal{E}}^{\mu,\nu}[-\rho_t(\xi)|{\cal{F}}_s],
\end{equation*}
which imply that both $\rho_t(\xi)$ and $-\rho_t(\xi)$ are $g^{\mu,\nu}$-submartingales. From Lemma \ref{dmg}, it follows that there exist $(Z^1_t,A^1_t)$ and $(Z^2_t,A^2_t)$ in $H_d^2\times\cal{A}$ such that
\begin{equation*}\label{4.9}
\rho_t(\xi)=-\xi+\int_t^Tg^{\mu,\nu}(Z_s^1)ds-A^1_T+A^1_t
-\int_t^TZ_s^1\cdot dB_s,\quad t\in[0,T],\tag{4.9}
\end{equation*}
and
\begin{equation*}\label{4.10}
\rho_t(\xi)=-\xi-\int_t^Tg^{\mu,\nu}(-Z_s^2)ds+A^2_T
-A^2_t-\int_t^T-Z_s^2\cdot dB_s,\quad t\in[0,T].\tag{4.10}
\end{equation*}
Comparing the martingale parts and the bounded variation parts of (\ref{4.9}) and (\ref{4.10}), we get that $dt\times dP$-$a.e.,$
\begin{align*}
  Z_t^1&=-Z_t^2,\\
g^{\mu,\nu}(Z_t^1)dt-dA^1_t&=-g^{\mu,\nu}(-Z_t^2)dt+dA^2_t,
\end{align*}
which imply that $dt\times dP$-$a.e.,$
\begin{equation*}
2g^{\mu,\nu}(Z_t^1)dt=dA^1_t+dA^2_t.
\end{equation*}
Since $A^1_t$ and $A^2_t$ are both increasing and continuous, by the Lebesgue decomposition for increasing functions, we can get that $A^1_t$ and $A^2_t$ are both absolutely continuous. Thus, we can find two nonnegative processes $a^1_t, a^2_t\in H_1^1$ such that $dt\times dP$-$a.e.,$  $dA^1_t=a^1_tdt$ and $dA^2_t=a^2_tdt.$ This implies $dt\times dP$-$a.e.,$
\begin{equation*}\label{4.11}
2g^{\mu,\nu}(|Z_t^1|)=a^1_t+a^2_t.\tag{4.11}
\end{equation*}
Set $Z^\xi_t:=Z_t^1$ and $g^\xi_t:=g^{\mu,\nu}(Z_t^{\xi})-a^1_t.$ From (\ref{4.9}) and (\ref{4.11}), it follows that (\ref{4.8}) holds and $dt\times dP$-$a.e.,$ $|g^\xi_t|\leq g^{\mu,\nu}(Z_t^{\xi})$.

Moreover, by Proposition \ref{rmp1}(ii) and (\hyperref[A]{A}), for $\xi, \eta\in\mathcal{L}^{\exp}_T$ and $\theta\in(0,1)$, we have
\begin{align*}
  \frac{\rho_s(\xi)-\theta\rho_s(\eta)}{1-\theta}
  &=\frac{\rho_s(-\rho_t(\xi))-\theta\rho_s(-\rho_t(\eta))}{1-\theta}\\
  &\leq\rho_s\left(\frac{-\rho_t(\xi)-\theta(-\rho_t(\eta))}{1-\theta}\right)\\
  &={\cal{E}}^{\mu,\nu}\left[\frac{\rho_t(\xi)-\theta \rho_t(\eta)}{1-\theta}|{\cal{F}}_s\right],\quad \forall 0\leq s\leq t\leq T.
\end{align*}
This means that $\frac{\rho_t(\xi)-\theta\rho_t(\eta)}{1-\theta}$ is a $g^{\mu,\nu}$-submartingale. Then by Lemma \ref{dmg} again, there exists a pair $(z_t,A_t)\in H_d^2\times\cal{A}$ such that
\begin{equation*}
\frac{\rho_t(\xi)-\theta\rho_t(\eta)}{1-\theta}=\frac{-\xi-\theta (-\eta)}{1-\theta}+\int_t^Tg^{\mu,\nu}(z_t)ds-A_T+A_t-\int_t^Tz_t\cdot dB_s,\quad t\in[0,T].
\end{equation*}
Moreover, by (\ref{4.8}), we have
\begin{equation*}
\frac{\rho_t(\xi)-\theta\rho_t(\eta)}{1-\theta}=\frac{-\xi-\theta (-\eta)}{1-\theta}+\int_t^T\left(\frac{g_s^\xi-\theta g_s^\eta}{1-\theta}\right)ds-\int_t^T\left(\frac{Z_s^\xi-\theta Z_s^\eta}{1-\theta}\right)\cdot dB_s,\quad t\in[0,T].
\end{equation*}
Comparing the martingale parts and the bounded variation parts of the two equations above, we get that $dt\times dP$-$a.e.,$
\begin{equation*}
\frac{g_t^\xi-\theta g_t^\xi}{1-\theta}\leq g^{\mu,\nu}(z_t)\quad \textrm{with}\quad z_t=\frac{Z_t^\xi-\theta Z_t^\eta}{1-\theta}.
\end{equation*}
The proof is complete.
\end{proof}

\begin{remark}\label{r4.4}
Let $\rho$ be a convex DRM satisfying (\hyperref[A]{A}). For $\xi\in\mathcal{L}^{\exp}_T$ and $\tau\in{\cal{T}}_{0,T}$, we set $\rho_\tau(\xi)(\omega):=\rho_{\tau(\omega)}(\xi)(\omega).$ Then, it is not hard to check that on the domain $\mathcal{L}^{\exp}_T$, all the results in (\hyperref[r1]{r1})-(\hyperref[r7]{r7}), Proposition \ref{rmp1}(i)(ii), Proposition \ref{rmp2}(ii) and (\ref{4.8}) still hold, when $t$ therein are replaced by any stopping time $\tau\in{\cal{T}}_{0,T}.$ We only show the case of Proposition \ref{rmp1}(i). For all $\tau\in{\cal{T}}_{0,T}$, we can find a decreasing
sequence $\{\tau_n\}_{n\geq1}\subset{\cal{T}}_{0,T}$, which all take values in some finite set such that $\tau_n\searrow\tau$ as $n\rightarrow\infty.$ We assume that $\tau_n$ take values: $t^n_{1},\cdots,t^n_{n_k}$, by Proposition \ref{rmp1}(i), for all $n\geq1$ and $\xi\in\mathcal{L}^{\exp}_T, \zeta\in\mathcal{L}^{\exp}_\tau$, we have
\begin{equation*}
\rho_{\tau_n}(\xi+\zeta)=\sum_{i=1}^{n_k}1_{\{\tau_n=t^n_{i}\}}\rho_{t^n_{i}}(\xi+\zeta)
=\sum_{i=1}^{n_k}1_{\{\tau_n=t^n_{i}\}}\rho_{t^n_{i}}(\xi)-\zeta=\rho_{\tau_n}(\xi)-\zeta.
\end{equation*}
Since all the paths of $\rho_t(\xi)$ are continuous, letting $n\rightarrow\infty,$ we obtain $\rho_{\tau}(\xi+\zeta)=\rho_{\tau}(\xi)-\zeta$.
\end{remark}
\section{A Doob-Meyer decomposition for $\rho$-supermartingales}
In this section, we will establish a decomposition of Doob-Meyer's type for $\rho$-supermartingales, which plays a key role in the proof of Theorem \ref{rmt}. We first introduce the notion of $\rho$-martingales (resp. $\rho$-supermartingales, $\rho$-submartingales).
\begin{definition}
Let $\rho$ be a convex DRM, and let $Y_t$ be a progressively measurable process such that for all $t\in [0,T],$ $Y_t\in\mathcal{L}^{\exp}_t$. If for all $0\leq s\leq t\leq T,$ we have $\rho_s(-Y_t)=Y_s$ (resp. $\leq,\ \geq$), then the process $Y_t$ is called a $\rho$-martingale (resp. $\rho$-supermartingale, $\rho$-submartingale).
\end{definition}
We have the following optional stopping theorem:
\begin{prop}\label{opt} Let $\rho$ be a convex DRM satisfying (\hyperref[A]{A}). Let $Y_t$ be a $\rho$-supermartingale (resp. $\rho$-submartingale). Then, for all $\sigma, \tau\in{\cal{T}}_{0,T},$
\begin{equation*}
\rho_\sigma(-Y_\tau)\leq Y_{\sigma\wedge\tau},\ (\textrm{resp.} \geq Y_{\sigma\wedge\tau}).
\end{equation*}
\end{prop}
\begin{proof} By Proposition \ref{rmp2}(ii) and the proof of \cite[Lemma 3.5]{CH}, for all $\xi,\eta\in\mathcal{L}^{\exp}_T$, $t\in[0,T]$ and $A\in{\mathcal {F}}_t,$ we have
\begin{equation*}\label{5.1}
\rho_t(1_A\xi+1_{A^c}\eta)=1_A\rho_t(\xi)+1_{A^c}\rho_t(\eta).\tag{5.1}
\end{equation*}
We only prove the case that $Y_t$ is a $\rho$-supermartingale. In this case, we have for all $0\leq s\leq t\leq T,$ $\mathcal{E}^{-\mu,-\nu}[Y_t]\leq\rho_s(-Y_t)\leq Y_s$. This implies that $Y_t$ is a $g^{-\mu,-\nu}$-supermartingale. Thus, each path of $Y_t$ is right continuous (see Remark \ref{rrcll}). In view of this, (\ref{5.1}) and the convergence in Proposition \ref{rmp2}(i), following the proof of \cite[Theorem 7.4]{Peng04}, we can obtain the desired result.
\end{proof}

We have the following well-posedness result:
\begin{prop}\label{pr5.3} Let $\rho$ be a convex DRM satisfying (\hyperref[A]{A}). Let $Y_t\in\mathcal{S}^{\infty}$, $z\in \mathbf{R}^d$, and $Y_t+z\cdot B_t$ be a $\rho$-supermartingale. Then, for any $n\geq1$, the following backward equation:
\begin{equation*}\label{5.2}
y_t^n=\rho_t\left(-Y_T-z\cdot (B_T-B_t)-\int_t^Tn(Y_s-y_s^n)ds\right),\quad t\in[0,T],\tag{5.2}
\end{equation*}
admits a unique solution $y_t^n\in{\cal{S}}^\infty$. Moreover, the following hold:

(i) For all $n\geq1$ and $t\in[0,T]$, $y_t^n\leq y_t^{n+1}\leq Y_t;$

(ii) For all $n\geq1$, $y^n_t+z\cdot B_t$ is a $\rho$-supermartingale.
\end{prop}
\begin{proof} \textbf{Step 1.} For $n\geq1$ and $y_t\in{\mathcal{S}}^\infty$, we set
\begin{equation*}
\phi(y_t):=\rho_t\left(-Y_T-z\cdot B_T-\int_0^Tn(Y_s-y_s)ds\right)-z\cdot B_t-\int_0^tn(Y_s-y_s)ds,\quad t\in[0,T].
\end{equation*}
By Proposition \ref{rmp2}(iii), for any $y_t, \tilde{y}_t\in{\mathcal{S}}^\infty$, we have $\phi(y_t), \phi(\tilde{y}_t)\in{\cal{S}}^\infty$ and by Proposition \ref{rmp1}(iii),
\begin{equation*}
\|\phi(y_t)-\phi(\tilde{y}_t)\|_{\mathcal{S}^\infty}\leq 2nT\|y_t-\tilde{y}_t\|_{\mathcal{S}^\infty}.
\end{equation*}
This implies that if $T\leq\frac{1}{4n}$, then $\phi(\cdot)$ is a strict contraction on $\mathcal{S}^\infty$. Thus, when $T\leq\frac{1}{4n}$, by Proposition \ref{rmp1}(i), we get that the equation (\ref{5.2}) admits a unique solution $y_t^n\in{\cal{S}}^\infty$.

When $T>\frac{1}{4n}$, we set $0=T_1<T_2<\cdots<T_m=T$ such that for any $1<k\leq m$, $T_k-T_{k-1}\leq\frac{1}{4n}$. Since for any $1<k\leq m$, the equation (\ref{5.2}) has a unique bounded solution on $[T_{k-1},T_k]$, using the ``patching-up" method in the proof of \cite[Proposition 4.4]{HM}, we conclude that the equation (\ref{5.2}) admits a unique solution $y_t^n\in{\cal{S}}^\infty$.

\textbf{Step 2.} For $n\geq1$, suppose that there exists ${t_0}\in[0,T)$ such that $P(y_{t_0}^n>Y_{t_0})>0$, which implies that there exists $k\geq1$ such that
\begin{equation*}\label{5.3}
P\left(y_{t_0}^n>Y_{t_0}+\frac{1}{k}\right)>0.\tag{5.3}
\end{equation*}
Set $A_1:=\{y_{t_0}^n>Y_{t_0}+\frac{1}{k}\}$ and
$\tau:=\{t\geq {t_0}:y_t^n\leq Y_t\}\wedge T$. Since $y_T^n=Y_T$, we have
\begin{equation*}\label{5.4}
1_{A_1}y_\tau^n=1_{A_1}Y_\tau\quad \text{and}\quad 1_{A_1}y_t^n\geq1_{A_1}Y_t,\ t\in[{t_0},\tau].\tag{5.4}
\end{equation*}
From Proposition \ref{rmp1}(i) and (\ref{5.2}), it follows that
\begin{equation*}\label{5.5}
y_t^n+z\cdot B_t+\int_0^tn(Y_s-y^n_s)ds=\rho_t\left(-Y_T-z\cdot B_T-\int_0^Tn(Y_s-y_s^n)ds\right),\quad\forall t\in[0,T],\tag{5.5}
\end{equation*}
which implies that $y_t^n+z\cdot B_t+\int_0^tn(Y_s-y^n_s)ds$ is a $\rho$-martingale. Then by Proposition \ref{opt}, Proposition \ref{rmp1}(i), Proposition \ref{rmp2}(ii), (\hyperref[r1]{r1}) and (\ref{5.4}), we have
\begin{align*}
1_{A_1}y_{t_0}^n&=1_{A_1}\rho_{t_0}\left(-y_\tau^n-z\cdot (B_\tau-B_{t_0})-\int_{t_0}^\tau n(Y_s-y_s^n)ds\right)\\
&=1_{A_1}\rho_{t_0}\left(-1_{A_1}y_\tau^n-1_{A_1}z\cdot (B_\tau-B_{t_0})-\int_{t_0}^\tau n(1_{A_1}Y_s-1_{A_1}y_s^n)ds\right)\\
&\leq1_{A_1}\rho_{t_0}\left(-1_{A_1}Y_\tau-1_{A_1}z\cdot (B_\tau-B_{t_0})\right)\\
&=1_{A_1}\rho_{t_0}\left(-Y_\tau-z\cdot B_\tau\right)-1_{A_1}z\cdot B_{t_0}\\
&\leq1_{A_1}Y_{t_0},
\end{align*}
which contradicts (\ref{5.3}). Thus, for all $n\geq1$ and $t\in[0,T]$, $y_t^n\leq Y_t.$

\textbf{Step 3.} For $n\geq1$, suppose that there exists ${t_0}\in[0,T)$ such that $P(y_{t_0}^n>y_{t_0}^{n+1})>0$, which implies that there exists $k\geq1$ such that
\begin{equation*}\label{5.6}
P\left(y_{t_0}^n>y_{t_0}^{n+1}+\frac{1}{k}\right)>0.\tag{5.6}
\end{equation*}
Set $A_2:=\{y_{t_0}^n>y_{t_0}^{n+1}+\frac{1}{k}\}$ and
$\sigma:=\{t\geq {t_0}:y_t^n\leq y_t^{n+1}\}\wedge T$. We then have
\begin{equation*}\label{5.7}
1_{A_2}y_\sigma^n=1_{A_2}y_\sigma^{n+1}\quad \text{and}\quad 1_{A_2}y_t^n\geq1_{A_2}y_t^{n+1},\ t\in[{t_0},\sigma].\tag{5.7}
\end{equation*}
By Proposition \ref{opt}, Proposition \ref{rmp1}(i), Proposition \ref{rmp2}(ii), (\ref{5.7}), the consequence of Step 2, (\hyperref[r1]{r1}), and (\ref{5.5}), we have
\begin{align*}
1_{A_2}y_{t_0}^n&=1_{A_2}\rho_{t_0}\left(-y_\sigma^n-z\cdot (B_\sigma-B_{t_0})-\int_{t_0}^\sigma n(Y_s-y_s^n)ds\right)\\
&=1_{A_2}\rho_{t_0}\left(-1_{A_2}y_\sigma^n-1_{A_2}z\cdot (B_\sigma-B_{t_0})-\int_{t_0}^\sigma n(1_{A_2}Y_s-1_{A_2}y_s^n)ds\right)\\
&\leq1_{A_2}\rho_{t_0}\left(-1_{A_2}y_\sigma^{n+1}-1_{A_2}z\cdot (B_\sigma-B_{t_0})-\int_{t_0}^\sigma(n+1)(1_{A_2}Y_s-1_{A_2}y_s^{n+1})ds\right)\\
&=1_{A_2}\rho_{t_0}\left(-y_\sigma^{n+1}-z\cdot (B_\sigma-B_{t_0})-\int_{t_0}^\sigma(n+1)(Y_s-y_s^{n+1})ds\right)\\
&=1_{A_2}y_{t_0}^{n+1},
\end{align*}
which contradicts (\ref{5.6}). Thus, for all $n\geq1$ and $t\in[0,T]$, $y_t^n\leq y_t^{n+1}.$ This with Step 2 implies (i).

\textbf{Step 4.} By (\ref{5.5}), Proposition \ref{rmp1}(i), the consequence of Step 2, and (\hyperref[r1]{r1}), we have
\begin{align*}
 y_s^n+z\cdot B_s&=\rho_s\left(-y_t^n-z\cdot B_t-\int_s^tn(Y_r-y_r^n)dr\right)\geq\rho_s(-y_t^n-z\cdot B_t),\quad \forall 0\leq s\leq t\leq T,
\end{align*}
which implies that $y_t^n+z\cdot B_t$ is a $\rho$-supermartingale. Thus, (ii) holds.
\end{proof}

We have the following decomposition theorem for $\rho$-supermartingales.
\begin{theorem} \label{dm} Let $\rho$ be a convex DRM satisfying (\hyperref[A]{A}). Let $Y_t\in\mathcal{S}^{\infty}$, $z\in \mathbf{R}^d$, and $Y_t+z\cdot B_t$ be a $\rho$-supermartingale. Then there exists a process $A_t\in {\cal{A}}$ and an increasing sequence $\{\tau_n\}_{n\geq1}\subset{\cal{T}}_{0,T}$ satisfying $\tau_n\nearrow T$ as $n\rightarrow\infty$, such that for all $n\geq1$, $Y_{\tau_n\wedge t}+z\cdot B_{\tau_n\wedge t}+A_{\tau_n\wedge t}$ is a $\rho$-martingale. Moreover, if $A_T\in L^\infty({\cal{F}}_T)$, then $Y_t+z\cdot B_t+A_t$ is a $\rho$-martingale.
\end{theorem}

\begin{proof} By Proposition \ref{pr5.3} and Proposition \ref{rmp1}(i), the following equation:
\begin{equation*}\label{5.8}
y_t^n+z\cdot B_t+\int_0^tn(Y_s-y_s^n)ds=\rho_t\left(-Y_T-z\cdot B_T-\int_0^Tn(Y_s-y_s^n)ds\right),\quad t\in[0,T],\tag{5.8}
\end{equation*}
admits a unique solution $y_t^n\in{\cal{S}}^\infty$ such that for all $n\geq1$ and $t\in[0,T]$, $y_t^n\leq y_t^{n+1}\leq Y_t.$ Thus, there exists a constant $C>0$ such that $dt\times dP$-$a.e.$,
\begin{equation*}\label{5.9}
\sup_{n\geq1}|y^n_t|<C.\tag{5.9}
\end{equation*}
For $n\geq1$, set
\begin{equation*}\label{5.10}
A_t^n:=\int_0^tn(Y_s-y_s^n)ds,\quad t\in[0,T].\tag{5.10}
\end{equation*}
This with (\ref{5.8}) implies that for all $t\in[0,T],$
\begin{equation*}\label{5.11}
\rho_t(-Y_T-z\cdot B_T-A_T^n)=y_t^n+z\cdot B_t+A_t^n.\tag{5.11}
\end{equation*}
From Proposition \ref{rp}, it follows that there exists a pair $(g^n_s, Z^n_s)\in H^1_1\times H_d^2$ such that
\begin{equation*}\label{5.12}
y_t^n+z\cdot B_t=Y_T+z\cdot B_T+A_T^n-A_t^n+\int_t^Tg^n_sds-\int_t^TZ^n_s\cdot dB_s,\quad t\in[0,T],\tag{5.12}
\end{equation*}
with
\begin{equation*}\label{5.13}
|g^n_t|\leq(\mu+\nu)(1+|Z^n_t|^2),\quad dt\times dP\text{-}a.e.,\tag{5.13}
\end{equation*}
and for $m\geq1$ and $\theta\in(0,1)$,
\begin{equation*}\label{5.14}
\frac{g_t^n-\theta g_t^m}{1-\theta}\leq(\mu+\nu)\left(1+\left|\frac{Z^n_t-\theta Z^m_t}{1-\theta}\right|^2\right),\quad dt\times dP\text{-}a.e.\tag{5.14}
\end{equation*}
Then we have
\begin{prop}\label{pro5.4}
\begin{equation*}\label{5.15}
\lim_{n\rightarrow\infty}E\left[\sup_{t\in[0,T]}|y^n_t-Y_t|\right]\rightarrow0,\tag{5.15}
\end{equation*}
\begin{equation*}\label{5.16}
\lim_{n,m\rightarrow\infty}E\left[\int_0^T|Z^n_t-Z^m_t|^2dt\right]\rightarrow0,\quad and \quad \lim_{n,m\rightarrow\infty}E\left[\int_0^T|g^n_t-g^m_t|dt\right]\rightarrow0.\tag{5.16}
\end{equation*}
\end{prop}

\begin{proof}
For $n\geq1$, $k>0$ and $\delta\in\mathcal{T}_{0,T},$ we set
\begin{equation*}
{\tau^k_\delta}:=\inf\left\{t\geq\delta:\int_\delta^t|Z^n_s-z|^2ds\geq k\right\}\wedge T.
\end{equation*}
For a constant $\beta<0$, applying It\^{o}'s formula to $e^{\beta y^n_t}$ and by (\ref{5.13}), we have
\begin{align*}\label{5.17}
e^{\beta y^n_\delta}&+\frac{1}{2}\beta^2\int_\delta^{\tau^k_\delta}e^{\beta y^n_s}|Z^n_s-z|^2ds-\beta\int_\delta^{\tau^k_\delta}e^{\beta y^n_s}dA_s\notag\\&\leq e^{\beta Y_{\tau^k_\delta}}+\int_\delta^{\tau^k_\delta}|\beta| e^{\beta y^n_s}|g^n_s|ds-\int_\delta^{\tau^k_\delta}\beta e^{\beta y^n_s}(Z^n_s-z)\cdot dB_s\tag{5.17}\\
&\leq e^{\beta Y_{\tau^k_\delta}}+(\mu+\nu)(1+2|z|^2)|\beta|\int_\delta^{\tau^k_\delta}e^{\beta y^n_s}ds+2(\mu+\nu)|\beta|\int_\delta^{\tau^k_\delta}e^{\beta y^n_s}|Z^n_s-z|^2ds\\
&\ \ \ \ \ \ -\int_\delta^{\tau^k_\delta}\beta e^{\beta y^n_s}(Z^n_s-z)\cdot dB_s,
\end{align*}
which together with (\ref{5.9}) implies
\begin{align*}\label{5.18}
\left(\frac{1}{2}\beta^2-2(\mu+\nu)|\beta|\right)&E\left[\int_\delta^{\tau^k_\delta}e^{\beta y^n_s}|Z^n_s-z|^2ds|{\cal{F}}_\delta\right]\leq e^{|\beta|C}(1+(\mu+\nu)(1+2|z|^2)|\beta|T).\tag{5.18}
\end{align*}
Choose $\beta<-4(\mu+\nu)$ such that $\frac{1}{2}\beta^2-2(\mu+\nu)|\beta|=1$. By (\ref{5.9}) and Fatou's lemma, when $k$ tends to $\infty$, we have
\begin{equation*}\label{5.19}
E\left[\int_\delta^T|Z^n_s-z|^2ds|{\cal{F}}_\delta\right]\leq e^{2|\beta|C}(1+(\mu+\nu)(1+2|z|^2)|\beta|T).\tag{5.19}
\end{equation*}
From this and the energy inequality for BMO martingales (see \cite[Page 29]{Ka}), we conclude that for all $m\geq1$,
\begin{equation*}\label{5.20}
\sup_{n\geq1}E\left[\left(\int_0^T|Z_s^n|^2ds\right)^m\right]\leq \sup_{n\geq1}m!\|Z^n_s\|^{2m}_{BMO}<\infty.\tag{5.20}
\end{equation*}

By (\ref{5.12}), (\ref{5.9}) and (\ref{5.13}), we have
\begin{align*}
|A_T^n|^2&\leq\left(2C+\int_0^T|g_s^n|ds+\left|\int_0^T(Z^n_s-z)\cdot dB_s\right|\right)^2\\
&\leq3\left(4C^2+2(\mu+\nu)^2T^2+2(\mu+\nu)^2\left(\int_0^T|Z_s^n|^2ds\right)^2
+\left|\int_0^T(Z^n_s-z)\cdot dB_s\right|^2\right).
\end{align*}
Taking expectation and then by BDG inequality and (\ref{5.20}), we have
\begin{equation*}\label{5.21}
\sup_{n\geq1}E|A_T^n|^2<\infty.\tag{5.21}
\end{equation*}
Then, by (\ref{5.9}), dominated convergence theorem, (\ref{5.10}) and (\ref{5.21}), we have
\begin{equation*}
E\int_0^T\lim_{n\rightarrow\infty}|Y_s-y_s^n|ds= \lim_{n\rightarrow\infty}E\int_0^T|Y_s-y_s^n|ds
\leq\lim_{n\rightarrow\infty}\frac{1}{n}\sup_{n\geq1}E\left[|A_T^n|\right]=0,
\end{equation*}
which implies that
\begin{equation*}\label{5.22}
y^n_t\nearrow Y_t,\ dt\times dP\text{-}a.e.,\quad \text{as}\quad n\rightarrow\infty.\tag{5.22}
\end{equation*}

By Proposition \ref{pr5.3}(ii) and (\hyperref[A]{A}), we have for all $0\leq s\leq t\leq T,$
\begin{align*}
 -y_s^n-z\cdot B_s\leq-\rho_s(-y_t^n-z\cdot B_t)\leq-\mathcal{E}^{-\mu,-\nu}[y_t^n+z\cdot B_t|\mathcal{F}_s]=\mathcal{E}^{\mu,\nu}[-y_t^n-z\cdot B_t|\mathcal{F}_s],
\end{align*}
which implies that $-y_t^n-z\cdot B_t$ is a $g^{\mu,\nu}$-submartingale. It follows from Lemma \ref{dmg} that there exists a pair $(\tilde{Z}^n_s,\tilde{A}_s)\in H^2_d\times\mathcal{A},$ such that
\begin{equation*}\label{5.23}
y_t^n+z\cdot B_t=Y_T+z\cdot B_T+\tilde{A}_T^n-\tilde{A}_t^n+\int_t^Tg^{-\mu,-\nu}(-\tilde{Z}_s^n)ds-\int_t^T-\tilde{Z}_s^n\cdot dB_s,\quad t\in[0,T],\tag{5.23}
\end{equation*}
Thus, $\tilde{A}_t^n$ is continuous. Comparing (\ref{5.23}) with (\ref{5.12}), we have $dt\times dP$-$a.e.$, $-\tilde{Z}_s^n=Z_s^n$. Applying It\^{o}'s formula to $-\exp(-2\nu(y_t^n+z\cdot B_t))$, we have
\begin{equation*}\label{5.24}
\tilde{y}_t^n=-\exp(-2\nu(Y_T+z\cdot B_T))+\tilde{a}_T^n-\tilde{a}_t^n+\int_t^T-\mu|\tilde{z}_s^n|ds
-\int_t^T\tilde{z}_s^n\cdot dB_s,\quad t\in[0,T],\tag{5.24}
\end{equation*}
with $\tilde{y}_t^n=-\exp(-2\nu(y_t^n+z\cdot B_t)),\ \tilde{z}_t^n=2\nu\exp(-2\nu(y_t^n+z\cdot B_t))Z_s^n$ and $ \tilde{a}_t^n=2\nu\int_0^t\exp(-2\nu(y_s^n+z\cdot B_s))d\tilde{A}^n_s$. Then by (\ref{5.9}), we have
\begin{align*}\label{5.25}
\sup_{n\geq1}E\left[\sup_{t\in[0,T]}|\tilde{y}_t^n|^2\right]
&=\sup_{n\geq1}E\left[\sup_{t\in[0,T]}\exp(-4\nu(y_t^n+z\cdot B_t))\right]\\
&\leq\exp(4\nu C)E\left[\sup_{t\in[0,T]}\exp(-4\nu z\cdot B_t))\right]<\infty.\tag{5.25}
\end{align*}
From this and Bouchard et al. \cite[Theorem 2.1]{Bou}, we further have
\begin{align*}\label{5.26}
\sup_{n\geq1}E\left(\int_0^T|\tilde{z}_s^n|^2ds+|\tilde{a}_T^n|^2\right)<\infty.\tag{5.26}
\end{align*}
By (\ref{5.9}), there exists a bounded progressively measurable process $\tilde{Y}_t$ such that for all $t\in[0,T]$, $y^n_t\nearrow \tilde{Y}_t$, as $n\rightarrow\infty$. This implies that for all $t\in[0,T]$, $\tilde{y}_t^n\nearrow -\exp(-2\nu(\tilde{Y}_t+z\cdot B_t))$, as $n\rightarrow\infty$. Then, in view of (\ref{5.24})-(\ref{5.26}), by \cite[Theorem 2.4]{Peng99}, we deduce that $-\exp(-2\nu(\tilde{Y}_t+z\cdot B_t))$ is RCLL and thus $\tilde{Y}_t$ is RCLL. This with (\ref{5.22}) implies that for all $t\in[0,T]$, $\tilde{Y}_t=Y_t$. Then by Dini's theorem, we have
\begin{equation*}
\lim_{n\rightarrow\infty}\sup_{t\in[0,T]}|y^n_t-Y_t|\rightarrow0.
\end{equation*}
In view of this and (\ref{5.9}), by dominated convergence theorem, we obtain (\ref{5.15}).

Applying It\^{o}'s formula to $|y^n_t-y^m_t|^2$ and (\ref{5.13}), we have
\begin{align*}
E&|y^n_0-y^m_0|^2+E\int_0^T|Z^n_s-Z^m_s|^2ds\\&=2E\int_0^T(y^n_s-y^m_s)(g^n_s-g^m_s)ds
+2E\int_0^T(y^n_s-y^m_s)d(A^n_s-A^m_s)\\
&\leq2(\mu+\nu+1)E\left[\sup_{t\in[0,T]}|y^n_t-y^m_t|\left(\int_0^T(2+|Z^n_s|^2+|Z^m_s|^2)ds
+(|A^n_T|+|A^m_T|)\right)\right]\\
&\leq2(\mu+\nu+1)\sqrt{E\left[\sup_{t\in[0,T]}|y^n_t-y^m_t|^2\right]}
\sqrt{2E\left[\left(\int_0^T(2+|Z^n_s|^2+|Z^m_s|^2)ds\right)^2\right]
+2E|A^n_T+A^m_T|^2},
\end{align*}
which with (\ref{5.15}), (\ref{5.20}) and (\ref{5.21}) implies that $\{Z^n_t\}_{n\geq1}$ is a Cauchy sequence in ${\cal{H}}^2_d.$

By (\ref{5.14}) and (\ref{5.13}), we have
\begin{align*}
g_t^n-g_t^m&=g_t^n-\theta g_t^m-(1-\theta)g_t^m\\
&\leq(\mu+\nu)(1-\theta)\left(1+\left|Z^n_t+\frac{\theta}{1-\theta}(Z^n_t-Z^m_t)\right|^2\right)+(1-\theta)|g_t^m|\\
&\leq(\mu+\nu)(1-\theta)\left(1+2|Z^n_t|^2+2\left(\frac{\theta}{1-\theta}\right)^2|Z^n_t-Z^m_t|^2\right)
+\mu(1-\theta)(1+|Z^m_t|^2)\\
&\leq2(\mu+\nu)(1-\theta)\left(1+|Z^n_t|^2+|Z^m_t|^2\right)
+\frac{2(\mu+\nu)\theta^2}{1-\theta}|Z^n_t-Z^m_t|^2.
\end{align*}
By interchanging $g_t^n$ and $g_t^m$, we deduce that
\begin{align*}
E\int_0^T|g_t^n-g_t^m|dt\leq4(\mu+\nu)(1-\theta)\sup_{n\geq1}E\int_0^T\left(1+|Z^n_t|^2\right)dt
+\frac{2(\mu+\nu)\theta^2}{1-\theta}E\int_0^T|Z^n_t-Z^m_t|^2dt.
\end{align*}
Since (\ref{5.20}) holds and $\{Z^n_t\}_{n\geq1}$ is a Cauchy sequence in ${\cal{H}}^2_d$, letting $n, m\rightarrow\infty$, and then $\theta\rightarrow1$, we get that $\{g^n_t\}_{n\geq1}$ is a Cauchy sequence in ${\cal{H}}^1_1.$ We obtain (\ref{5.16}).
\end{proof}

By (\ref{5.12}) and Proposition \ref{pro5.4}, we deduce that for all $t\in[0,T]$, $\{A^n_t\}_{n\geq1}$ is a Cauchy sequence in $L^1({\cal{F}}_t).$ We denote the limit of $\{A^n_t\}_{n\geq1}$ in $L^1({\cal{F}}_t)$ by $A_t$. Then we deduce that $A_t$ is an increasing and continuous process with $A_0=0$.

Since $\{Z^n_t\}_{n\geq1}$ is a Cauchy sequence in ${\cal{H}}^2_d$ and $\{\int_0^TZ^n_t\cdot dB_t\}_{n\geq1}$ is a Cauchy sequence in $L^2({\cal{F}}_T)$, by Kobylanski \cite[Lemma 2.5]{K}, we can find a subsequence of $\{Z^n_t\}_{n\geq1}$ still denoted by $\{Z^n_t\}_{n\geq1}$ such that
$\sup_{n\geq1}|Z_t^n|\in{\cal{H}}^2_d$ and $\sup_{n\geq1}\left|\int_0^TZ^n_t\cdot dB_t\right|\in L^2({\cal{F}}_T)$. From this, (\ref{5.12}), (\ref{5.9}) and (\ref{5.13}), we have
\begin{align*}
E[\sup_{n\geq1}A_T^n]&\leq2C+E\left[\int_0^T\sup_{n\geq1}|g_s^n|ds\right]
+E\left[\sup_{n\geq1}\left|\int_0^TZ^n_s\cdot dB_s\right|\right]\\
&\leq2C+(\mu+\nu)E\left[\int_0^T\sup_{n\geq1}(1+|Z^n_s|^2)ds\right]
+E\left[\sup_{n\geq1}\left|\int_0^TZ^n_s\cdot dB_s\right|\right]<\infty.
\end{align*}
Set
\begin{equation*}
  \sigma_k:=\inf\left\{t\geq0,E\left[\sup_{n\geq 1}A_T^n|{\cal{F}}_t\right]\geq k\right\}\wedge T, \ \ k>0.
\end{equation*}
By (\ref{5.11}) and Proposition \ref{opt}, we get that for all $t\in[0,T]$ and $k>0,$
\begin{equation*}
\rho_t(-y^n_{\sigma_k}-z\cdot B_{\sigma_k}-A^n_{\sigma_k})=y^n_{{\sigma_k}\wedge t}+z\cdot B_{{\sigma_k}\wedge t}+A^n_{{\sigma_k}\wedge t}.
\end{equation*}
Since $A^n_t$ is increasing, we have $\sup_{n\geq 1}A^n_{\sigma_k}\leq E\left[\sup_{n\geq 1}A_T^n|{\cal{F}}_{\sigma_k}\right]\leq k$. Hence, by (\ref{5.9}), (\ref{5.15}) and Proposition \ref{rmp2}(i), we get that for all $t\in[0,T]$ and $k>0,$
\begin{equation*}
\rho_t(-Y_{\sigma_k}-z\cdot B_{\sigma_k}-A_{\sigma_k})=Y_{{\sigma_k}\wedge t}+z\cdot B_{{\sigma_k}\wedge t}+A_{{\sigma_k}\wedge t}.
\end{equation*}
Moreover, if $A_T$ is bounded, then letting $k\rightarrow\infty$, by Proposition \ref{rmp2}(i) again, we get that for all $t\in[0,T],$
\begin{equation*}
\rho_t(-Y_T-z\cdot B_T-A_T)=Y_t+z\cdot B_t+A_t.
\end{equation*}
The proof is complete.
\end{proof}
\section{The proofs of Theorem \ref{rmt} and Corollary \ref{rmt2}}\label{sec6}
\begin{proof}[\emph{\textbf{The proof of Theorem \ref{rmt}:}}]

\textbf{Step 1.} For $z\in \textbf{R}^d$, we have
\begin{equation*}\label{6.1}
-g^{\mu,\nu}(z)t+z\cdot B_t=-g^{\mu,\nu}(z)T+z\cdot B_T+\int_t^Tg^{\mu,\nu}(z)ds-\int_t^Tz\cdot dB_s,\quad t\in[0,T],\tag{6.1}
\end{equation*}
which implies that $-g^{\mu,\nu}(z)t+z\cdot B_t$ is a $g^{\mu,\nu}$-martingale. From (\hyperref[A]{A}), it follows that $-g^{\mu,\nu}(z)t+z\cdot B_t$ is a $\rho$-supermartingale. Then by Theorem \ref{dm}, there exists a process $A_t^z\in {\cal{A}}$ and an increasing series $\{\tau_n\}_{n\geq1}\subset\mathcal{T}_{0,T}$ satisfying $\tau_n\nearrow T$ as $n\rightarrow\infty$, such that for all $n\geq1$ and $t\in[0,T],$
\begin{equation*}
\rho_t(g^{\mu,\nu}(z)\tau_n-z\cdot B_{\tau_n}-A_{ \tau_n}^z)=-g^{\mu,\nu}(z)({t\wedge\tau_n})+z B_{t\wedge\tau_n}+A_{t\wedge\tau_n}^z.
\end{equation*}
From Proposition \ref{rp}, it follows that for all $n\geq1$, there exists a pair $(g_n(s,z), Z^{z,n}_s)\in H_1^1\times H_d^2$ such that
\begin{equation*}\label{6.2}
-g^{\mu,\nu}(z)({t\wedge\tau_n})+z\cdot B_{t\wedge\tau_n}+A_{t\wedge\tau_n}^z
=-g^{\mu,\nu}(z){\tau_n}+z\cdot B_{\tau_n}+A_{\tau_n}^z+\int_{t\wedge\tau_n}^{\tau_n}g_n(s,z)ds-\int_{t\wedge\tau_n}^{\tau_n}Z^{z,n}_s\cdot dB_s.\tag{6.2}
\end{equation*}
Comparing (\ref{6.1}) and (\ref{6.2}), we have for all $n\geq1$, $dt\times dP$-$a.e.,$
\begin{equation*}\label{6.3}
Z^{z,n}_t=z\quad \text{and} \quad |g_n(t,z)|\leq\mu|z|+\nu|z|^2,\quad t\in[0,\tau_n],\tag{6.3}
\end{equation*}
and thus,
\begin{equation*}
A_{\tau_n}^z=\int_0^{\tau_n}g^{\mu,\nu}(z)ds-\int_0^{\tau_n} g_n(s,z)ds\leq2T(\mu|z|+\nu|z|^2)<\infty,
\end{equation*}
which implies that $A_T^z$ is bounded. Then by Theorem \ref{dm} again, we get that for all $t\in[0,T],$
\begin{equation*}\label{6.4}
\rho_t(g^{\mu,\nu}(z)T-z\cdot B_T-A_T^z)=-g^{\mu,\nu}(z)t+z\cdot B_t+A_t^z.\tag{6.4}
\end{equation*}
By Proposition \ref{rp} again and a similar argument as in (\ref{6.3}), we deduce that there exists a process $g(t,z)\in H_1^1$ such that
\begin{equation*}\label{6.5}
-g^{\mu,\nu}(z)t+z\cdot B_t+A_t^z
=-g^{\mu,\nu}(z)T+z\cdot B_T+A_T^z+\int_t^Tg(s,z)ds-\int_t^Tz\cdot dB_s,\tag{6.5}
\end{equation*}
where $dt\times dP$-$a.e.,$ $|g(t,z)|\leq\mu|z|+\nu|z|^2$ and for all $\tilde{z}\in\textbf{R}^d$ and $\theta\in(0,1)$,
\begin{equation*}
g(t,z)-\theta g(t,\tilde{z})\leq\mu|z-\theta\tilde{z}|
+\frac{\nu}{1-\theta}|z-\theta\tilde{z}|^2.
\end{equation*}
This implies that $g\in{\cal{G}}^{\mu,\nu}\cap{\cal{G}}^{\mu,\nu}_{\theta}$. Moreover, by (\ref{6.5}), Proposition \ref{rmp1}(i), (\ref{6.4}) and Proposition \ref{opt}, we get that for all $z\in\mathbf{R}^d$, $\tau\in\mathcal{T}_{0,T}$ and $0\leq r\leq t\leq T$,
\begin{align*}
&\rho_{\tau\wedge r}\left(\int_{\tau\wedge r}^{\tau\wedge t} g(s,z)ds-\int_{\tau\wedge r}^{\tau\wedge t} z\cdot dB_s\right)\\
&=\rho_{\tau\wedge r}(g^{\mu,\nu}(z)({\tau\wedge t})-z\cdot B_{\tau\wedge t}^z-A_{\tau\wedge t}^z-(g^{\mu,\nu}(z)({\tau\wedge r})-z\cdot B_{\tau\wedge r}^z-A_{\tau\wedge r}^z))\\
&=\rho_{\tau\wedge r}(g^{\mu,\nu}(z)({\tau\wedge t})-z\cdot B_{\tau\wedge t}^z-A_{\tau\wedge t}^z)+(g^{\mu,\nu}(z)({\tau\wedge r})-z\cdot B_{\tau\wedge r}^z-A_{\tau\wedge r}^z)\\
&=0.
\end{align*}

\textbf{Step 2.} In view of the consequence of Step 1 and Proposition \ref{rmp2}(ii), by similar arguments as in \cite[(4.8)]{Zheng24}, for any $\tau\in\mathcal{T}_{0,T}$, $r\in[0,T]$ and ${\cal{F}}_{\tau\wedge r}$-measurable simple random variable $\eta$ (i.e., $\eta=\sum^{m}_{j=1}z_{j}1_{A_j}$ where for all $1\leq j\leq m$, $z_{j}\in{\bf{R}}^d$ and $\{A_{j}\}_{j=1}^{m}$ is a partition of $\Omega$ such that for all $1\leq j\leq m$, $A_{j}\in{\cal{F}}_{\tau\wedge r}$), we have
\begin{align*}\label{6.6}
\rho_{\tau\wedge r}\left(\int_{\tau\wedge r}^{\tau\wedge t} g(s,\eta)ds-\int_{\tau\wedge r}^{\tau\wedge t} \eta\cdot dB_s\right)=0,\quad \forall t\geq r.\tag{6.6}
\end{align*}
Let ${\cal{O}}$ be the set of all the simple step processes defined as
\begin{equation*}
\eta_t=\sum^l_{i=0}\sum^{m_i}_{j=1}z_{ij}1_{[t_i,t_{i+1})\times A_{ij}}(t,\omega),\ \ (t,\omega)\in[0,T]\times\Omega,
\end{equation*}
where $0=t_0<t_1<\cdots<t_l<t_{l+1}=T$, $z_{ij}\in{\bf{R}}^d$ and $\{A_{ij}\}_{j=1}^{m_i}$ is a partition of $\Omega$ such that for all $0\leq i\leq l$ and $0\leq j\leq m_i$, $A_{ij}\in{\cal{F}}_{t_i}$. For any $\eta_s\in{\cal{O}}$, $\tau\in\mathcal{T}_{0,T}$ and $0\leq r\leq T$ with $r\in[t_i,t_{i+1})$ for some $0\leq i\leq l$, by setting $r_i=r$ and $r_k=t_k,$ $i<k\leq l+1$, we have
\begin{align*}\label{6.7}
&\rho_{\tau\wedge r}\left(\int_{\tau\wedge r}^{\tau}g(s,\eta_s)ds-\int_{\tau\wedge r}^{\tau}\eta_s\cdot dB_s\right)\\
&=\rho_{\tau\wedge r}\left(\sum^l_{k=i}\left(\int_{\tau\wedge r_k}^{\tau\wedge r_{k+1}}g(s,\eta_s)ds
-\int_{\tau\wedge r_k}^{\tau\wedge r_{k+1}}\eta_s\cdot dB_s\right)\right)\\
&=\rho_{\tau\wedge r}\left(-\rho_{\tau\wedge r_l}\left(\int_{\tau\wedge r_l}^{\tau}g(s,\eta_{\tau\wedge r_l})ds
-\int_{\tau\wedge r_l}^{\tau}\eta_{\tau\wedge r_l}\cdot dB_s\right)\right.\\
&\ \ \ \ \ \left.+\sum^{l-1}_{k=i}\left(\int_{\tau\wedge r_k}^{\tau\wedge r_{k+1}}g(s,\eta_s)ds
-\int_{\tau\wedge r_k}^{\tau\wedge r_{k+1}}\eta_s\cdot dB_s\right)\right)\quad \quad \text{by\ (\hyperref[r2]{r2})} \text{\ and\ Proposition\ \ref{rmp1}(i)}\\
&=\rho_{\tau\wedge r}\left(\sum^{l-1}_{k=i}\left(\int_{\tau\wedge r_k}^{\tau\wedge r_{k+1}}g(s,\eta_s)ds
-\int_{\tau\wedge r_k}^{\tau\wedge r_{k+1}}\eta_s\cdot dB_s\right)\right)\quad \quad \text{by}\ (\ref{6.6})\\
&\cdots\\
&=\rho_{\tau\wedge r}\left(\int_{\tau\wedge r}^{\tau\wedge r_{i+1}}g(s,\eta_{\tau\wedge r})ds
-\int_{\tau\wedge r}^{\tau\wedge r_{i+1}}\eta_{\tau\wedge r}\cdot dB_s\right)\\
&=0.\quad \quad \text{by}\ (\ref{6.6})\tag{6.7}
\end{align*}

Since $g(t,z)\in{\cal{G}}^{\mu,\nu}\cap{\cal{G}}^{\mu,\nu}_{\theta}$, by \cite[Theorem 2.6(i)]{FHT20}, for $\xi\in\mathcal{L}^{\exp}_T,$ the following BSDE:
\begin{equation*}\label{6.8}
Y_t=-\xi+\int_t^Tg(s,Z_s)ds-\int_t^TZ_s\cdot dB_s,\quad t\in[0,T].\tag{6.8}
\end{equation*}
has a unique solution $(Y_t,Z_t)\in\mathcal{L}^{\exp}_\mathcal{F}\times \bigcap_{r>1}\mathcal{H}_d^r.$
It is clear that there exists a sequence $\{Z^n_t\}_{n\geq1}\subset{\cal{O}}$ such that $Z^n_t\rightarrow Z_t$ in ${\cal{H}}^2_d$ as $n\rightarrow\infty.$ Since
\begin{equation*}
\left|\sup_{t\in[0,T]}\left|\int_0^tZ^n_s\cdot dB_s\right|-\sup_{t\in[0,T]}\left|\int_0^tZ_s\cdot dB_s\right|\right|
\leq\sup_{t\in[0,T]}\left|\int_0^t(Z^n_s-Z_s)\cdot dB_s\right|,
\end{equation*}
by BDG inequality, we have $\sup_{t\in[0,T]}\left|\int_0^tZ^n_s\cdot dB_s\right|\rightarrow\sup_{t\in[0,T]}\left|\int_0^tZ_s\cdot dB_s\right|$ in $L^2({\cal{F}}_T)$ as $n\rightarrow\infty.$ Then by \cite[Lemma 2.5]{K}, we can find a subsequence of $\{Z^n_t\}_{n\geq1}$ still denoted by $\{Z^n_t\}_{n\geq1}$ such that $\sup_{n\geq 1}|Z^n_t|\in{\cal{H}}^2_d$ and
$\sup_{n\geq 1}\sup_{t\in[0,T]}\left|\int_0^tZ^n_s\cdot dB_s\right|\in L^2({\cal{F}}_T).$

For $k>0$, we define
\begin{equation*}
\sigma_k:=\inf\left\{t\geq0:\int_0^t\sup_{n\geq1}|Z^n_s|^2ds
+\sup_{n\geq1}\sup_{r\in[0,t]}\left|\int_0^rZ^n_s\cdot dB_s\right|\geq k\right\}\wedge T.
\end{equation*}
By D\'{e}but theorem, we get that for any $k>0$, $\sigma_k$ is a stopping time. Since $g(t,z)\in{\cal{G}}^{\mu,\nu}\cap{\cal{G}}^{\mu,\nu}_{\theta}$, by Remark \ref{2.1} and dominated convergence theorem, we deduce that for all $k>0$, there exists a subsequence of $\{Z^n_t\}_{n\geq1}$ still denoted by $\{Z^n_t\}_{n\geq1}$ such that
\begin{align*}
\int^{\sigma_k}_{\sigma_k\wedge t}g(s,Z^n_s)ds-\int^{\sigma_k}_{\sigma_k\wedge t}Z^n_s\cdot dB_s\rightarrow\int^{\sigma_k}_{\sigma_k\wedge t}g(s,Z_s)ds-\int^{\sigma_k}_{\sigma_k\wedge t}Z_s\cdot dB_s,\quad dP\text{-}a.s.,
\end{align*}
as $\rightarrow\infty.$ Moreover, since $\sup_{r\in[0,t]}\left|\int_0^rZ^n_s\cdot dB_s\right|$ is continuous in $t$ and $\sup_{n\geq1}\sup_{r\in[0,t]}\left|\int_0^rZ^n_s\cdot dB_s\right|$ is increasing in $t$, we can deduce that $\sup_{n\geq1}\sup_{r\in[0,t]}\left|\int_0^rZ^n_s\cdot dB_s\right|$ is left-continuous in $t$. Thus, for all $k>0$, we have
 \begin{align*}
 &\sup_{n\geq1}\left|\int^{\sigma_k}_{\sigma_k\wedge t}g(s,Z^n_s) ds+\int^{\sigma_k}_{\sigma_k\wedge t}Z^n_s\cdot dB_s\right|\\
 &\leq2\int^{\sigma_k}_0(\mu+\nu)(1+\sup_{n\geq1}|Z^n_s|^2)ds+2\sup_{n\geq1}\sup_{r\in[0,\sigma_k]}
 \left|\int^r_0Z^n_s\cdot dB_s\right|\\
 &\leq2(\mu+\nu)(T+k)+2k.
\end{align*}
By the two equations above, (\ref{6.7}) and Proposition \ref{rmp2}(i), we get that for all $k>0$,
\begin{equation*}
\rho_{\sigma_k\wedge t}\left(\int^{\sigma_k}_{\sigma_k\wedge t}g(s,Z_s)ds-\int^{\sigma_k}_{\sigma_k\wedge t}Z_s\cdot dB_s\right)=0.
\end{equation*}
Then by (\ref{6.8}), Proposition \ref{rmp1}(i) and Definition \ref{qg}(i), we get that for all $k>0$ and $t\in[0,T]$,
\begin{equation*}
\rho_{\sigma_k\wedge t}(-Y_{\sigma_k})=Y_{\sigma_k\wedge t}=\mathcal{E}^g[-\xi|{\cal{F}}_{\sigma_k\wedge t}].
\end{equation*}
From this and Proposition \ref{rmp2}(i), we deduce that for all $\xi\in\mathcal{L}^{\exp}_T,$ (\ref{3.2}) holds. Moreover, by Lemma \ref{g}(vii)(viii), we further get that $g$ is a unique function in ${\cal{G}}^{\mu,\nu}\cap{\cal{G}}_{conv}$ such that for all $\xi\in\mathcal{L}^{\exp}_T,$ (\ref{3.2}) holds.

\textbf{Step 3.} In view of (\ref{3.2}), to prove (\ref{3.3}), we only need to show that for all $\xi\in\mathcal{L}^{\exp}_T,$
\begin{align*}\label{6.9}
\mathcal{E}^g[-\xi|{\cal{F}}_t]
=\operatorname{ess\,sup}_{q\in\mathcal{Q}_{\xi,f}}E_{Q^q}\left[-\xi-\int_t^Tf(s,q_s)ds|\mathcal{F}_t\right],\quad \forall t\in[0,T],\tag{6.9}
\end{align*}
where $\mathcal{Q}_{\xi,f}:=\{q_t\in H_d^2: \theta^q_t:=\exp\{\int_0^tq_s\cdot dB_s-\frac{1}{2}\int_0^t|q_s|^2ds\}$ is a uniformly integrable martingale such that $E_{Q^q}[|\xi|+\int_0^T|f(s,q_s)|ds]<\infty$ with $dQ^q=\theta^q_TdP\}.$
(\ref{6.9}) can be derived from the proof of \cite[Theorem 3.1(i)]{FHT24}. For convenience, we show its proof in the \hyperref[appA]{Appendix}, where we can also find that the supremum in (\ref{3.3}) is attained by any $q_t\in H^2_d$ such that $dt\times dP$-$a.e.,$ $q_t\in \partial g(t,Z_t^{-\xi,g})$. The proof is complete.
\end{proof}

\begin{proof} [\emph{\textbf{The proof of Corollary \ref{rmt2}:}}]
\textbf{Step 1.} By Theorem \ref{rmt}, there exists a unique $g\in{\cal{G}}^{\mu,\nu}\cap{\cal{G}}_{conv}$ such that for all $\xi\in\mathcal{L}^{\exp}_T$,
 \begin{equation*}\label{6.10}
 \rho_t(\xi)={\cal{E}}^g[-\xi|{\cal{F}}_t],\quad t\in[0,T].\tag{6.10}
\end{equation*}
Since $\rho$ satisfies (\hyperref[r5]{r5}) and (\hyperref[r5]{r6}) on $\mathcal{L}^{\exp}_T$, by Lemma \ref{g}(ix), we get that $g\in{\cal{G}}^{\mu,\nu}\cap{\cal{G}}_{subl}$. This implies that $dt\times dP$-$a.e.$, for all $z\in\mathbf{R}^d$ and $\beta>0$, $\beta|g(t,z)|=|g(t,\beta z)|\leq\beta\mu|z|+\beta^2\nu|z|^2$, and thus \begin{equation*}
|g(t,z)|\leq\mu|z|+\beta\nu|z|^2.
\end{equation*}
Letting $\beta\rightarrow0$, we have $g\in{\cal{G}}^{\mu,0}\cap{\cal{G}}_{subl}$.

Since $g\in{\cal{G}}^{\mu,0}$, we have, $dt\times dP$-$a.e.$,
\begin{equation*}\label{6.11}
f(t,x):=\sup_{z\in\mathbf{R}^d}\{x\cdot z-g(t,z)\}
\geq\sup_{z\in\mathbf{R}^d}\{x\cdot z-\mu|z|\}=\infty1_{\{|x|>\mu\}},\quad \forall x\in\mathbf{R}^d.\tag{6.11}
\end{equation*}
For $x\in\mathbf{R}^d$, if for all $z\in\mathbf{R}^d$, $x\cdot z\leq g(t,z)$, which is equivalent to $x\in\partial g(t,0)$ (since $g(t,0)\equiv0$), then we have $dt\times dP$-$a.e.$,
\begin{equation*}\label{6.12}
f(t,x)=\sup_{z\in\mathbf{R}^d}\{x\cdot z-g(t,z)\}=0.\tag{6.12}
\end{equation*}
and if for some $z\in\mathbf{R}^d$, $x\cdot z>g(t,z)$, then since $g\in{\cal{G}}_{subl}$, we have, $dt\times dP$-$a.e.$, for all $\beta>0$,
\begin{equation*}\label{6.13}
f(t,x)\geq \beta x\cdot z-g(t,\beta z)
=\beta(x\cdot z-g(t,z))\rightarrow\infty,\quad \text{as}\ \beta\rightarrow\infty.\tag{6.13}
\end{equation*}
By (\ref{6.11})-(\ref{6.13}), we get that $dt\times dP$-$a.e.$,
\begin{equation*}\label{6.14}
f(t,x)=\infty1_{\{x\notin \partial g(t,0)\}},\quad \forall x\in\mathbf{R}^d.\tag{6.14}
\end{equation*}
and for all $x\in\partial g(t,0)$, $|x|\leq\mu.$

By \cite[Theorem 2.6(i)]{FHT20}, for $\xi\in\mathcal{L}^{\exp}_T,$ the following BSDE:
\begin{equation*}\label{6.15}
Y_t=-\xi+\int_t^Tg(s,Z_s)ds-\int_t^TZ_s\cdot dB_s,\quad t\in[0,T].\tag{6.15}
\end{equation*}
has a unique solution $(Y_t,Z_t)\in\mathcal{L}^{\exp}_\mathcal{F}\times \bigcap_{r>1}\mathcal{H}_d^r.$
By measurable selection theorem (or \cite[Lemma 7.5]{BE}), we can find a progressively measurable process $\tilde{q}_t\in\partial g(t,Z_t)$, $dt\times dP$-$a.e.$ This implies that $dt\times dP$-$a.e.,$ $f(t,\tilde{q}_t)=\tilde{q}_t\cdot Z_t-g(t,Z_t),$ and then by (\ref{6.14}), we have $dt\times dP$-$a.e.,$ $\tilde{q}_t\in\partial g(t,0)$, $|\tilde{q}_t|\leq\mu$ and $f(t,\tilde{q}_t)=0$. Set $dQ^{\tilde{q}}/dP:=\exp(\int_0^T\tilde{q}_s\cdot dB_s-\frac{1}{2}\int_0^T|\tilde{q}_s|^2ds)$. Clearly, $Q^{\tilde{q}}\in\mathcal{Q}_g$. Since $dt\times dP$-$a.e.,$ $|\tilde{q}_t|\leq\mu$, we can deduce that for all $p\geq2$, $dQ^{\tilde{q}}/dP\in L^p(\mathcal{F}_T)$. Thus, by H\"{o}lder's inequality, we have for all $t\in[0,T]$, $E_{Q^{\tilde{q}}}[|Y_t|]<\infty$ and then by (\ref{6.15}), we have
\begin{align*}\label{6.16}
Y_t&=E_{Q^{\tilde{q}}}\left[-\xi+\int_t^Tg(s,Z_s)ds-\int_t^TZ_s\cdot dB_s|\mathcal{F}_t\right]\\
&=E_{Q^{\tilde{q}}}\left[-\xi-\int_t^TZ_s\cdot dB^{\tilde{q}}_s|\mathcal{F}_t\right]\\
&=E_{Q^{\tilde{q}}}[-\xi|\mathcal{F}_t],\quad \forall t\in[0,T],\tag{6.16}
\end{align*}
where
$B^{\tilde{q}}_t:=B_t-\int_0^t{\tilde{q}}_sds$ is a standard Brownian motion under the probability $Q^{\tilde{q}}$.

For any $q_t\in\mathcal{Q}_g,$ by (\ref{6.14}) and (\ref{6.11}), we have $dt\times dP$-$a.e.,$ $0=f(t,q_t)\geq q_t\cdot Z_t-g(t,Z_t).$ Then by a similar argument as in (\ref{6.16}), we have
\begin{align*}
Y_t\geq E_{Q^q}[-\xi|\mathcal{F}_t],\quad t\in[0,T].
\end{align*}
From this, (\ref{6.16}) and Definition \ref{qg}(i), it follows that for all $\xi\in\mathcal{L}^{\exp}_T$,
\begin{equation*}\label{6.17}
{\cal{E}}^g[-\xi|{\cal{F}}_t]=\sup_{Q\in\mathcal{Q}_g}E_{Q}[-\xi|\mathcal{F}_t],\quad t\in[0,T].\tag{6.17}
\end{equation*}
This with (\ref{6.10}) gives (\ref{3.4}).

If there exists another $\tilde{g}\in {\cal{G}}^{\mu,0}\cap{\cal{G}}_{subl}$ such that (\ref{3.4}) holds, then by the same arguments as in (\ref{6.11})-(\ref{6.17}), we deduce that for all $\xi\in\mathcal{L}^{\exp}_T$ and $t\in[0,T]$,
\begin{equation*}
{\cal{E}}^{\tilde{g}}[-\xi|{\cal{F}}_t]=\sup_{Q\in\mathcal{Q}_{\tilde{g}}}E_{Q}[-\xi|\mathcal{F}_t]
=\rho_t(\xi)=\sup_{Q\in\mathcal{Q}_g}E_{Q}[-\xi|\mathcal{F}_t]={\cal{E}}^g[-\xi|{\cal{F}}_t].
\end{equation*}
This with Lemma \ref{g}(vi) implies that $dt\times dP$-$a.e.,$ for all $z\in\mathbf{R}^d$, $g(t,z)=\tilde{g}(t,z).$

\textbf{Step 2.} Let us further assume that for all $\xi\in L^2(\mathcal{F}_T)$, $\rho$ satisfies (\ref{3.1}) with $\nu=0$. Then, by Proposition \ref{rmp1}(ii) and (\hyperref[A]{A}), for all $\xi, \eta\in L^2(\mathcal{F}_T)$ and $\theta\in(0,1)$, we have
\begin{align*}
  \rho_t(\xi)-\theta\rho_t(\eta)&\leq(1-\theta)\mathcal{E}^{\mu,0}\left[\frac{-\xi+\theta\eta}{1-\theta}|\mathcal{F}_t\right]\\
  &=\mathcal{E}^{\mu,0}\left[-\xi+\theta\eta|\mathcal{F}_t\right],\quad \forall t\in[0,T].
\end{align*}
Letting $\theta\rightarrow1$, by \cite[Theorem 3.2]{Peng04} and Lemma \ref{g}(vi), we have
\begin{equation*}
\rho_t(\xi)-\rho_t(\eta)\leq\mathcal{E}^{\mu,0}[\eta-\xi]\leq\mathcal{E}^{\mu,0}[|\xi-\eta|],\quad t\in[0,T].
\end{equation*}
By interchanging $\xi$ and $\eta$, we have
\begin{align*}\label{6.18}
|\rho_t(\xi)-\rho_t(\eta)|\leq\mathcal{E}^{\mu,0}[|\xi-\eta|],\quad \forall t\in[0,T].\tag{6.18}
\end{align*}
For $\xi\in L^2(\mathcal{F}_T)$, we set $\xi_n:=(\xi\vee(-n))\wedge n$. Then, by Step 1, we have $\rho_t(\xi_n)=\mathcal{E}^g[-\xi_n|\mathcal{F}_t]$. This with (\ref{6.18}) and \cite[Theorem 3.2]{Peng04} implies that for all $\xi\in L^2(\mathcal{F}_T)$, (\ref{3.2}) holds.

From the same arguments as in Step 1, we get that for all $\xi\in L^2(\mathcal{F}_T)$, (\ref{3.4}) holds. The proof is complete.
\end{proof}
\begin{appendix}
\section*{Appendix: The proof of (\ref{6.9})}\label{appA}
\begin{proof}
By the definition of $f$, we have, $dt\times dP$-$a.e.,$
\begin{equation*}\label{a.1}
f(t,x)\geq\sup_{z\in\mathbf{R}^d}\{x\cdot z-\mu|z|-\nu|z|^2\}\geq\frac{(|x|-\mu)^2}{4\nu}1_{\{|x|>\mu\}}
\geq\frac{1}{8\nu}(|x|^2-2\mu^2),\quad \forall x\in\mathbf{R}^d.\tag{a.1}
\end{equation*}
By \cite[Theorem 2.6(i)]{FHT20}, for $\xi\in\mathcal{L}^{\exp}_T,$ the following BSDE:
\begin{equation*}\label{a.2}
Y_t=-\xi+\int_t^Tg(s,Z_s)ds-\int_t^TZ_s\cdot dB_s,\quad t\in[0,T].\tag{a.2}
\end{equation*}
has a unique solution $(Y_t,Z_t)\in\mathcal{L}^{\exp}_\mathcal{F}\times \bigcap_{r>1}\mathcal{H}_d^r.$
By measurable selection theorem, we can find a progressively measurable process $\tilde{q}_t\in\partial g(t,Z_t)$, $dt\times dP$-$a.e.$ (see also \cite[Lemma 7.5]{BE}). This implies that $dt\times dP\text{-}a.e.,$
\begin{equation*}\label{a.3}
f(t,\tilde{q}_t)=\tilde{q}_t\cdot Z_t-g(t,Z_t).\tag{a.3}
\end{equation*}
By (\ref{a.1}) and (\ref{a.3}), we get that $dt\times dP\text{-}a.e.,$
\begin{align*}\label{a.4}
  \frac{1}{8\nu}(|\tilde{q}_t|^2-2\mu^2)\leq f(t,\tilde{q}_t)
  &\leq |\tilde{q}_t\cdot Z_t|+|g(t,Z_t)|\\
  &\leq16\nu|Z_t|^2+\frac{|\tilde{q}_t|^2}{16\nu}+(\mu+\nu)(1+|Z_t|^2),\tag{a.4}
\end{align*}
where implies that $\tilde{q}_t\in H_d^2$.

For any $q_t\in H_d^2$, we set
\begin{equation*}
\theta_t^q:=\exp\left(\int_0^tq_s\cdot dB_s-\frac{1}{2}\int_0^t|q_s|^2ds\right),\quad \forall t\in[0,T].
\end{equation*}
We define the following stopping time:
\begin{equation*}
\delta_n:=\inf\left\{s\geq0:\int_0^s|\tilde{q}_r|^2dr\geq n\right\}\wedge T.
\end{equation*}
and set $\tilde{q}_t^n:=\tilde{q}_t1_{\{t\in[0,\delta_n]\}}$. By \cite[Proposition 4.3 and Inequality (4.6)]{FHT24}, we have
\begin{equation*}\label{a.5}
E[\theta^{\tilde{q}^n}_T\ln(1+\theta^{\tilde{q}^n}_T)]\leq \frac{1}{2}E\left[\int_0^T\theta^{\tilde{q}^n}_s|\tilde{q}_s^n|^2ds\right]+\ln2,\tag{a.5}
\end{equation*}
and
\begin{equation*}\label{a.6}
|x||y|\leq\beta\exp\left(\frac{|x|}{\beta}\right)+\beta|y|\ln(1+|y|),\quad \forall x,y\in\mathbf{R}^d, \ \forall \beta>0.\tag{a.6}
\end{equation*}
By (\ref{a.5}), (\ref{a.1})-(\ref{a.3}) and (\ref{a.6}), we have
\begin{align*}\label{a.7}
\frac{1}{4\nu}(E[\theta^{\tilde{q}^n}_T\ln(1+\theta^{\tilde{q}^n}_T)]-\ln2)
-\frac{\mu^2T}{4\nu}
&\leq\frac{1}{8\nu}E\left[\int_0^T 1_{\{s\in[0,\delta_n]\}}\theta^{\tilde{q}^n}_s(|\tilde{q}_s^n|^2-2\mu^2)ds\right]\\
&\leq E\left[\int_0^T1_{\{s\in[0,\delta_n]\}}\theta^{\tilde{q}^n}_sf(s,\tilde{q}_s^n)ds\right]\\
&\leq E\left[\int_0^T1_{\{s\in[0,\delta_n]\}}E[\theta^{\tilde{q}^n}_T|\mathcal{F}_s]({\tilde{q}_s^n}\cdot Z_s-g(s,Z_s))ds\right]\\
&=E\left[\theta^{\tilde{q}^n}_T\int_0^T1_{\{s\in[0,\delta_n]\}}({\tilde{q}_s^n}\cdot Z_s-g(s,Z_s))ds\right]\\
&=E\left[\theta^{\tilde{q}^n}_T\left(Y_{\delta_n}-Y_0-\int_0^{\delta_n}Z_s\cdot dB^{\tilde{q}^n}_s\right)\right]\\
&\leq \frac{1}{8\nu}E\left[\sup_{t\in[0,T]}\exp(16\nu|Y_t|)\right]
+\frac{1}{8\nu}E[\theta^{\tilde{q}^n}_T\ln(1+\theta^{\tilde{q}^n}_T)],\tag{a.7}
\end{align*}
where
$B^{\tilde{q}^n}_t:=B_t-\int_0^t{\tilde{q}^n}_sds$ is a standard Brownian motion under the probability $Q^{\tilde{q}^n}$ satisfying $Q^{\tilde{q}^n}/dP=\theta^{\tilde{q}^n}_T$.
It follows from (\ref{a.7}) that there exists a constant $C$ independent of $n$ such that  $E[\theta^{\tilde{q}^n}_T\ln(1+\theta^{\tilde{q}^n}_T)]\leq C.$ Then by Protter \cite[Theorem 11]{Pro}, we get that $\{\theta^{\tilde{q}^n}_T\}_{n\geq1}$ is uniformly integrable, and then by Vitali convergence theorem, we have $E[\theta^{\tilde{q}}_T]=1$, which implies that $\theta^{\tilde{q}}_t$ is a uniformly integrable martingale. Moreover, by Fatou's lemma, we also have
\begin{equation*}\label{a.8}
  E[\theta^{\tilde{q}}_T\ln(1+\theta^{\tilde{q}}_T)]\leq C.\tag{a.8}
\end{equation*}

Set
\begin{equation*}
\tau_n^t:=\inf\left\{s\geq t:\int_t^s|Z_r|^2dr\geq n\right\}\wedge T.
\end{equation*}
By (\ref{a.2}) and (\ref{a.3}), we have
\begin{align*}\label{a.9}
Y_t&=E_{Q^{\tilde{q}}}\left[Y_{\tau_n^t}+\int_t^{\tau_n^t}g(s,Z_s)
-\int_t^{\tau_n^t}Z_s\cdot dB_s|\mathcal{F}_t\right]\\
&=E_{Q^{\tilde{q}}}\left[Y_{\tau_n^t}-\int_t^{\tau_n^t}f(s,\tilde{q}_s)
-\int_t^{\tau_n^t}Z_s\cdot dB^{\tilde{q}}_s|\mathcal{F}_t\right]\\
&=E_{Q^{\tilde{q}}}\left[Y_{\tau_n^t}-\int_t^{\tau_n^t}f(s,\tilde{q}_s)ds|\mathcal{F}_t\right],\quad \forall t\in[0,T],\tag{a.9}
\end{align*}
where $B^{\tilde{q}}_t:=B_t-\int_0^t{\tilde{q}}_sds$ is standard Brownian motion under the probability $Q^{\tilde{q}}$. By (\ref{a.6}) and (\ref{a.8}), we have
\begin{equation*}\label{a.10}
E_{Q^{\tilde{q}}}\left[\sup_{t\in[0,T]}|Y_t|\right]\leq
E\left[\sup_{t\in[0,T]}\exp(|Y_t|)\right]
+E[\theta^{\tilde{q}}_T\ln(1+\theta^{\tilde{q}}_T)]<\infty.\tag{a.10}
\end{equation*}
By (\ref{a.9}), (\ref{a.10}), Fatou's lemma and the fact that $f\geq0$, we have
\begin{equation*}\label{a.11}
E_{Q^{\tilde{q}}}\left[\int_t^T|f(s,\tilde{q}_s)|ds|\mathcal{F}_t\right]<\infty,\quad \forall t\in[0,T].\tag{a.11}
\end{equation*}
From (\ref{a.10}) and (\ref{a.11}), it follows that $\tilde{q}_t\in\mathcal{Q}_{(\xi,f)}$. Then, by (\ref{a.9})-(\ref{a.11}) and domination convergence theorem, we have
\begin{align*}\label{a.12}
Y_t=E_{Q^{\tilde{q}}}\left[-\xi-\int_t^Tf(s,\tilde{q}_s)ds|\mathcal{F}_t\right],\quad \forall t\in[0,T].\tag{a.12}
\end{align*}

We next show that for any $q_t\in\mathcal{Q}_{(\xi,f)},$
\begin{align*}\label{a.13}
Y_t\geq E_{Q^{q}}\left[-\xi-\int_t^Tf(s,q_s)ds|\mathcal{F}_t\right],\quad \forall t\in[0,T].\tag{a.13}
\end{align*}
By (\ref{a.5}), (\ref{a.1}), (\ref{a.3}) and the definition of $\mathcal{Q}_{(\xi,f)},$ we have
\begin{align*}\label{a.14}
\frac{1}{4\nu}(E[\theta^{q}_T\ln(1+\theta^{q}_T)]-\ln2)
-\frac{\mu^2T}{4\nu}
&\leq\frac{1}{8\nu}E\left[\int_0^T\theta^{q}_s(|q_s|^2-2\mu^2)ds\right]\\
&=E\left[\int_0^TE[\theta^{q}_T|\mathcal{F}_s]f(s,q_s)ds\right]\\
&=E_{Q^q}\left[\int_0^Tf(s,q_s)ds\right]<\infty,\tag{a.14}
\end{align*}
Since $f(t,q_t)\geq q_t\cdot Z_t-g(t,Z_t),$ by (\ref{a.14}) and similar arguments as in (\ref{a.9})-(\ref{a.10}), we can obtain (\ref{a.13}), which together with (\ref{a.12}) and Definition \ref{qg}(i) gives (\ref{6.9}).
\end{proof}
\end{appendix}
\
\\
\textbf{Declarations}
\\  \\
\textbf{Competing interests}\quad The authors declare no competing interests

\end{document}